\documentclass[11pt,a4paper]{article}
\usepackage[utf8]{inputenc}

\usepackage[margin=1in]{geometry}

\usepackage{amsmath}
\usepackage{amsfonts}
\usepackage{amssymb}
\usepackage{amsthm}
\usepackage{graphicx}
\usepackage{color}
\usepackage{comment}

\newtheorem{lem}{Lemma}
\newtheorem{remark}{Remark}
\newtheorem{prop}{Proposition}
\newtheorem{cor}{Corollary}
\newtheorem{thm}{Theorem}
\newtheorem{defn}{Definition}

\numberwithin{equation}{section}

\title{Distribution-Constrained Optimal Stopping}

\author{
Erhan Bayraktar\thanks{Department of Mathematics, University of Michigan ({erhan@umich.edu}). Supported in part by the NSF under grant number DMS-1613170 and in part by the Susan M. Smith Professorship.}
\and
 Christopher W. Miller\thanks{Department of Mathematics, University of California, Berkeley ({miller@math.berkeley.edu}). Supported in part by NSF GRFP under grant number DGE 1106400.}
}

\date{}

\begin{document}

\maketitle

\begin{abstract}We solve the problem of optimal stopping of a Brownian motion subject to the constraint that the stopping time's distribution is a given measure consisting of finitely-many atoms. In particular, we show that this problem can be converted to a finite sequence of state-constrained optimal control problems with additional states corresponding to the conditional probability of stopping at each possible terminal time. The proof of this correspondence relies on a new variation of the dynamic programming principle for state-constrained problems which avoids measurable selections. We emphasize that distribution constraints lead to novel and interesting mathematical problems on their own, but also demonstrate an application in mathematical finance to model-free superhedging with an outlook on volatility. \end{abstract}





\pagestyle{myheadings}
\thispagestyle{plain}
\markboth{E. Bayraktar and C. Miller}{Distribution-Constrained Optimal Stopping}

\section{Introduction}

In this paper we consider the problem of choosing an optimal stopping time for a Brownian motion when constrained in the choice of distribution for the stopping time. While standard optimal stopping theory has focused primarily on unconstrained finite- and infinite-horizon stopping times (e.g. \cite{PeskirShiryaev2006,Shiryaev2008}) and very recently on constraints on the first moment of the stopping time (e.g. \cite{Miller2015,PedersenPeskir2016,AnkirchnerKleinKruse2015}), there is a very limited literature on the problem of optimal stopping under distribution constraints.

It turns out that distribution-constrained optimal stopping is a difficult problem, with stopping strategies depending path-wise on the Brownian motion in general. This is to be expected because a constraint on the stopping time's distribution forces the stopper to consider what he would have done along all other paths of the Brownian motion when deciding whether to stop. The main task at hand is to identify relevant state variables and then transform the problem so that it can be analyzed by standard methods.

In this article we illustrate a solution in the special case that the target distribution consists of finitely-many atoms. Our approach consists of iterated stochastic control problems wherein we introduce controlled processes representing the conditional distribution of the stopping time. We then characterize the value function of the distribution-constrained optimal stopping problem in terms of the value functions of a finite number of state-constrained optimal control problems. This dynamic approach to the problem in terms of a controlled process with unbounded diffusion is similar in flavour to recent results in non-linear optimal stopping \cite{Miller2015} and control of measure-valued martingales in \cite{CoxKallblad2015}.

The key mathematical contributions of this paper lie in our proof of a dynamic programming principle relating each of the sequential optimal control problems. We provide an argument which avoids the use of measurable selections, similar to the proofs of weak dynamic programming principles in \cite{BouchardTouzi2011,BouchardNutz2012,BayraktarYao2013}. However, we deal with state-constraints in a novel way which relies on some a priori regularity of the value functions.

While the problem of distribution-constrained optimal stopping is mathematically-interesting in its own right, we emphasize that there is room for applications in mathematical finance and optimal control theory. For instance, we demonstrate an application to model-free superhedging of financial derivatives when one has an outlook on the quadratic variation of an asset price. Here, the distribution on the quadratic variation corresponds to that of a stopping time by the martingale time-change methods utilized recently in \cite{BonnansTan2013,GalichonLabordereTouzi2014}. Furthermore, the problem of optimal stopping under moment constraints on the stopping time reduces to the distribution-constrained optimal stopping problem in cases where there exists a unique atomic representing measure in the truncated moment problem (e.g. \cite{CurtoFialkow1991,Lasserre2010}). There also appears to be a connection between distribution-constrained optimal stopping and inverse first passage-time problems (e.g. \cite{Zucca2009,Capocelli1972}). We should also mention that after the publication of our preprint, \cite{2016arXiv161201488B} gave geometric descriptions of optimal stopping times using optimal transport theory.

This paper proceeds as follows. In Section \ref{Section:MainResults}, we provide our solution to distribution-constrained optimal stopping of Brownian motion. In particular, we characterize the solution via a finite sequence of iterated state-constrained stochastic control problems. The main result is provided by an induction argument in Theorem \ref{Thm:MainResult}, but the heart of the argument lies mainly in the proofs of Lemma~\ref{Lem:WeakDPPLemma} and Lemma~\ref{Lem:TerminalRelaxation}. We also provide a time-dependent versions of these results, which can be characterized as the viscosity solutions of associated HJB equations. The key arguments here lie in a Dynamic Programming Principle in Theorem~\ref{Thm:TimeDependentDPP}. In Section \ref{Section:Application}, we demonstrate an application to model-free superhedging with an outlook on volatility. We convert this problem into a distribution-constrained optimal stopping problem where the volatility outlook corresponds to a distribution constraint for the stopping time. We demonstrate numerical results which provide some intuition for the behavior of the optimal stopping strategies. Finally, we provide complete proofs of our main results in Appendices \ref{Appendix:WeakDPPLemma}--\ref{Appendix:TimeDependentDPP}.

\section{Main Results}\label{Section:MainResults}

\subsection{Problem formulation}

We consider a probability space $(\Omega,\mathcal{F},\mathbb{P})$ supporting a standard Brownian motion $W$. We take $\mathbb{F}:=\{\mathcal{F}_t\}_{t\geq 0}$ to be the natural filtration of $W$ augmented to satisfy the usual properties. We consider a given payoff function $f:\mathbb{R}\to\mathbb{R}$ which is assumed to be Lipschitz continuous. We also use the notation
\[X^{t,x}_u := x + W_u-W_t\]
for any $(x,t)\in\mathbb{R}\times[0,\infty)$ and $u\in[t,\infty)$.

In this paper, we are also given a target distribution $\mu$, which is supported on $(0,\infty)$ and assumed to consist of finitely-many atoms. Without loss of generality, we assume the following representation
\begin{equation}\label{Eqn:AtomicMeasure}
\mu = \sum_{k=1}^r p_k\delta_{t_k},
\end{equation}
where $r\in\mathbb{N}$, $0=t_0<t_1<\cdots<t_r$, $p_1+\cdots+p_r=1$, and $p_1,\ldots,p_r>0$. We also introduce the convenient notation $\Delta t_k := t_k-t_{k-1}$ for each $k\in\{1,\ldots,r\}$.

The distribution-constrained optimal stopping problem we consider is
\begin{equation}\label{Equation:MainProblem}\begin{array}{rcl}
v^\star := & \sup\limits_{\tau\in\mathcal{T}} & \mathbb{E}\left[f(X^{0,x_0}_\tau)\right] \\
& \text{s.t.} & \tau \sim \sum_{k=1}^r p_k\delta_{t_k},
\end{array}\end{equation}
where we take $\mathcal{T}$ to be the collection of all finite-valued $\mathbb{F}$-stopping times which are independent of $\mathcal{F}_0$. We let $x_0\in\mathbb{R}$ be some fixed starting value. That is, we choose a stopping time $\tau$ whose distribution is equal to $\mu$ in order to maximize the expected payoff of a stopped Brownian motion starting at $x_0$.

%

\subsection{Construction of distribution-constrained stopping times}

There are multiple ways to naturally represent a stopping time satisfying a distribution constraint. In this section, we outline two particular such representations and illustrate how they immediately lead to constructions of such stopping times.

We first provide a characterization of distribution-constrained stopping times in terms of a partitioning of the path space into regions with specified measures. Later, we make a connection with controlled processes.

\begin{lem}\label{Lem:PartitionCharacterization}
A stopping time $\tau$ has the distribution $\mu$ if and only if it is of the following form
\begin{equation}\nonumber
\tau = \sum_{k=1}^r t_k\,1_{A_k},
\end{equation}
almost-surely, where $\{A_1,\ldots,A_r\}$ partition $\Omega$ and, for each $k\in\{1,\ldots,r\}$, $A_k$ is $\mathcal{F}_{t_k}$-measurable with $\mathbb{P}\left[A_k\right]=p_k$.
\end{lem}

\begin{proof}

It is clear from the construction that such a $\tau$ is a $\mathbb{F}$-stopping time and $\tau\sim\mu$. The converse follows by taking a stopping time $\tau$ such that $\tau\sim\mu$ and defining the sets $A_k := \{\tau=t_k\}$ for each $k\in\{1,\ldots,r\}$.

\end{proof}

With this in mind, we can immediately explicitly construct a stopping time with given distribution.

\begin{cor}\label{Cor:ExistenceConstruction}
There exists a stopping time $\tau$ such that $\tau\sim\mu$.
\end{cor}
\begin{proof}
Define a partition $\{A_1,\ldots,A_r\}$ of $\Omega$ as
\begin{eqnarray}
A_1 & := & \left\{W_{t_1}-W_0 \leq \sqrt{t_1}\,\Phi^{-1}\left(p_1\right)\right\} \nonumber\\
A_2 & := & \left\{W_{t_2}-W_{t_1} \leq \sqrt{t_2-t_1}\,\Phi^{-1}\left(\frac{p_2}{p_2+\cdots+p_r}\right)\right\} \setminus A_1 \nonumber\\
& \vdots \nonumber\\
A_k & := & \left\{W_{t_k}-W_{t_{k-1}} \leq \sqrt{t_k-t_{k-1}}\,\Phi^{-1}\left(\frac{p_k}{p_k+\cdots +p_r}\right)\right\} \setminus \left( A_1\cup\cdots\cup A_{k-1} \right) \nonumber\\
& \vdots \nonumber\\
A_r & := & \Omega \setminus \left( A_1 \cup \cdots \cup A_{r-1} \right),\nonumber
\end{eqnarray}
where $\Phi$ is the cumulative distribution function of the standard normal distribution. It is clear that $A_k$ is $\mathcal{F}_{t_k}$-measurable with $\mathbb{P}\left[A_k\right]=p_k$ for each $k\in\{1,\ldots,r\}$. Then, by Lemma \ref{Lem:PartitionCharacterization}, $\tau := \sum_{k=1}^r t_k\,1_{A_k}$ defines a stopping time with $\tau\sim\mu$.
\end{proof}

The proof above constructs a stopping time which roughly stops when there are events in the left-tail of a distribution. However, one could easily modify the construction to stop in right-tail events, events near the median, or on the image of any Borel set of appropriate measure under $\Phi$.

While this construction may suggest converting the distribution-constrained optimal stopping problem into optimization over Borel sets of specified measure, we emphasize next that there is no reason to expect the stopping times to be measurable with respect to $\sigma(W_{t_1},\ldots,W_{t_r})$. In particular, in the next example, we show a construction of a distribution-constrained stopping time which is entirely path-dependent.

\begin{cor}\label{Cor:ExistenceConstruction2}
There exists a stopping time $\tau$, independent of $(W_{t_1},\ldots,W_{t_r})$, satisfying $\tau\sim\mu$.
\end{cor}
\begin{proof}
Define a sequence of random variables $(M_1,\ldots,M_r)$ as
\begin{equation}\nonumber
M_k := \left(t_k-t_{k-1}\right)^{-1/2}\max\limits_{t_{k-1}\leq s\leq t_k} \left\vert W_s - W_{t_{k-1}}-(s-t_{k-1})\frac{W_{t_k}-W_{t_{k-1}}}{t_k-t_{k-1}} \right\rvert
\end{equation}
for each $k\in\{1,\ldots,r\}$. Then each $M_k$ is the absolute maximum of a Brownian bridge over $[t_{k-1},t_k]$, scaled by the length of the time interval. In particular, each $M_k$ is $\mathcal{F}_{t_k}$-measurable, independent of $(W_{t_1},\ldots,W_{t_r})$, and equal in distribution to the absolute maximum of a standard Brownian bridge on $[0,1]$, the cumulative distribution function of which we denote by $\Phi_{BB}$.

Define a partition $\{A_1,\ldots,A_r\}$ of $\Omega$ as
\begin{eqnarray}
A_1 & := & \left\{M_1 \leq \Phi_{BB}^{-1}\left(p_1\right)\right\} \nonumber\\
A_2 & := & \left\{M_2 \leq \Phi_{BB}^{-1}\left(\frac{p_2}{p_2+\cdots +p_r}\right)\right\} \setminus A_1 \nonumber\\
& \vdots \nonumber\\
A_k & := & \left\{M_k \leq \Phi_{BB}^{-1}\left(\frac{p_k}{p_k+\cdots+p_r}\right)\right\} \setminus \left( A_1\cup\cdots\cup A_{k-1} \right) \nonumber\\
& \vdots \nonumber\\
A_r & := & \Omega \setminus \left( A_1 \cup \cdots \cup A_{r-1} \right).\nonumber
\end{eqnarray}
It is clear that $A_k$ is $\mathcal{F}_{t_k}$-measurable with $\mathbb{P}\left[A_k\right]=p_k$ for each $k\in\{1,\ldots,r\}$. Then, by Lemma \ref{Lem:PartitionCharacterization}, $\tau := \sum_{k=1}^r t_k\,1_{A_k}$ defines a stopping time with $\tau\sim\mu$ which is independent of $(W_{t_1},\ldots,W_{t_r})$.
\end{proof}

Clearly, the stopping time constructed above is an admissible stopping time in the distribution-constrained optimal stopping problem, but there is no hope to express it in terms of the value of the Brownian motion at each potential time to stop. While stopping times involving the Brownian bridge may seem unnatural at first, their use is a key idea in the proofs of Lemma \ref{Lem:WeakDPPLemma} and Lemma \ref{Lem:TerminalRelaxation}.

Another useful result obtained from Lemma~\ref{Lem:PartitionCharacterization} is the following approximation result.
\begin{prop}\label{Proposition:ApproximateStoppingTimeConstruction}
Fix $p,p'\in[0,1]^r$ satisfying $p_1+\cdots+p_r=1$ and $p'_1+\cdots +p'_r=1$. For any $\tau\in\mathcal{T}$ such that
\[\tau \sim \sum\limits_{k=1}^r p_k\delta_{t_k},\]
there exists $\tau'\in\mathcal{T}$ such that
\[\tau' \sim \sum\limits_{k=1}^r p'_k\delta_{t_k},\]
which satisfies
\[\mathbb{P}\left[\tau\neq\tau'\right] \leq 4^r \|p-p'\|_{\ell^1}.\]
\end{prop}

\begin{proof}
\textit{Step 1: }By Lemma~\ref{Lem:PartitionCharacterization}, there exists a partition $\{A_1,\ldots,A_r\}$ of $\Omega$ such that each $A_k$ is $\mathcal{F}_{t_k}$-measurable with $\mathbb{P}\left[A_k\right] = p_k$ and
\[\tau = \sum\limits_{k=1}^r t_k 1_{A_k},\]
almost-surely. The goal is to define $\tau'$ in terms of a related partition.

We first make a key observation: for any $t > 0$, $\theta\in(0,1)$, and $\mathcal{F}_t$-measurable set $A$ such that $\mathbb{P}\left[A\right]>0$, there exists a real number $w\in\mathbb{R}$ such that
\[\mathbb{P}\left[W_t \geq w \mid A\right] = \theta.\]
This follows immediately from the observation that the distribution of $W_t$ has no atoms, and thus when conditioning on an event of non-zero probability, the conditional distribution cannot have atoms.

\vspace{5mm}\textit{Step 2:} Define an $\mathcal{F}_{t_1}$-measurable set $A_1'$ as
\[A_1' := \left\{\begin{array}{cl}
A_1 & p_1 = p_1' \\
A_1\cup\{W_{t_1} \geq w^+_1\} & p_1 < p_1' \\
A_1\cap\{W_{t_1} \geq w^-_1\} &  p_1 > p_1'
\end{array}\right.\]
where $w^+_1,w^-_1\in\mathbb{R}$ are chosen such that
\[\mathbb{P}\left[W_{t_1}\geq w^+_1 \mid \Omega\setminus A_1\right] = \frac{p_1'-p_1}{1-p_1},\qquad\mathbb{P}\left[W_{t_1}\geq w^-_1 \mid A_1\right]=\frac{p_1'}{p_1}.\]
Notice, if $p_1 < p_1'$, then $p_1 < 1$ and $\mathbb{P}\left[\Omega\setminus A_1\right]>0$, so $w^+_1$ is well-defined. Similarly, if $p_1>p_1'$, then $p_1>0$ and $\mathbb{P}\left[A_1\right]>0$, so $w^-_1$ is well-defined.

It is clear that $\mathbb{P}\left[A_1'\right] = p_1'$ by construction. The key property, however, is that either $A_1\subseteq A_1'$ or $A_1'\subseteq A_1$. From this, we can immediately compute the measure of the symmetric difference of $A_1$ and $A_1'$,
\[\mathbb{P}\left[A_1 \triangle A'_1\right] = \left|p_1-p_1'\right|.\]

\vspace{5mm}\textit{Step 3:} Now, suppose that we have constructed $\{A'_1,\ldots,A'_{k-1}\}$ already. Define
\[q_k := \mathbb{P}\left[A_k\setminus(A'_1\cup\cdots\cup A'_{k-1})\right].\]
Define a $\mathcal{F}_{t_k}$-measurable set $A'_k$ as
\[A_k' := \left\{\begin{array}{cl}
A_k\setminus(A'_1\cup\cdots\cup A'_{k-1}) & q_k = p_k' \\
(A_k\setminus(A'_1\cup\cdots\cup A'_{k-1}))\cup\{W_{t_k} \geq w^+_k\} & q_k < p_k' \\
(A_k\setminus(A'_1\cup\cdots\cup A'_{k-1}))\cap\{W_{t_k} \geq w^-_k\} &  q_k > p_k'
\end{array}\right.\]
where $w^+_k,w^-_k\in\mathbb{R}$ are chosen such that
\[\begin{array}{l}
\mathbb{P}\left[W_{t_k}\geq w^+_k \mid (A'_1\cup\cdots\cup A'_{k-1})\cup(\Omega\setminus A_k)\right] = \frac{p_k'-q_k}{1-q_k}, \\ 
\mathbb{P}\left[W_{t_k}\geq w^-_k \mid A_k\setminus(A'_1\cup\cdots\cup A'_{k-1})\right]=\frac{p_k'}{q_k}.
\end{array}\]
As before, the inequalities between $q_k$ and $p_k'$ imply that $w^+_k$ and $w^-_k$ are well-defined when they are needed. Furthermore, it is clear that $\mathbb{P}\left[A_k'\right] = p_k'$.

In this case, the key property becomes that either
\begin{equation}\label{Equation:ApproximationCase1}
A_k\setminus(A'_1\cup\cdots\cup A'_{k-1}) \subseteq A_k'
\end{equation}
or
\begin{equation}\label{Equation:ApproximationCase2}
A_k' \subseteq A_k\setminus(A'_1\cup\cdots\cup A'_{k-1}) \subseteq A_k.
\end{equation}
First, we consider the case that \eqref{Equation:ApproximationCase1} holds. Because each set in $\{A_1,\ldots,A_k\}$ is disjoint, for any $\ell\in\{1,\ldots,k-1\}$, we can bound the overlap between $A_k$ and $A'_\ell$ by the symmetric difference of $A_\ell$ and $A'_\ell$,
\[\mathbb{P}\left[A_k\cap A'_\ell\right] \leq \mathbb{P}\left[ A'_\ell \setminus A_\ell \right] \leq \mathbb{P}\left[A_\ell \triangle A'_\ell\right].\]
Using \eqref{Equation:ApproximationCase1}, we can compute
\[\mathbb{P}\left[A_k\triangle A'_k\right] = \mathbb{P}\left[A'_k\right] - \mathbb{P}\left[A_k\right] + 2\mathbb{P}\left[A_k\setminus A'_k\right] \leq \left|p_k-p'_k\right| + 2\mathbb{P}\left[A_k\cap(A'_1\cup\cdots\cup A'_{k-1})\right].\]
Using the previous inequalities we deduce
\[\mathbb{P}\left[ A_k \triangle A'_k \right] \leq \left|p_k-p_k'\right| + \sum\limits_{\ell=1}^{k-1} 2\mathbb{P}\left[A_k\cap A'_\ell\right] \leq \left|p_k-p_k'\right| + \sum\limits_{\ell=1}^{k-1} 2\mathbb{P}\left[A_\ell\triangle A'_\ell\right].\]
In the case that \eqref{Equation:ApproximationCase2} holds, this same inequality immediately follows.

\vspace{5mm}\textit{Step 4:} By induction on $k\in\{1,\ldots,r\}$, we construct a disjoint partition $\{A'_1,\ldots,A'_r\}$ of $\Omega$ such that $A'_k$ is $\mathcal{F}_{t_k}$-measurable and $\mathbb{P}\left[A'_k\right]=p'_k$. Furthermore, we obtain a rough bound
\[\mathbb{P}\left[A_k\triangle A'_k\right] \leq \sum\limits_{\ell=1}^k 4^{k-\ell}\left|p_\ell-p'_\ell\right|.\]
By Lemma~\ref{Lem:PartitionCharacterization}, there exists a stopping time $\tau'$ of the form
\[\tau' = \sum\limits_{k=1}^r t_k 1_{A'_k}.\]
Then we can immediately compute\footnote{We emphasize that the preceding inequalities are hardly sharp. An immediate question is whether a sharper approximation result, with a constant that scales favorably as $r\to\infty$, exists.}
\[\mathbb{P}\left[\tau\neq\tau'\right]\leq\sum\limits_{k=1}^r \mathbb{P}\left[A_k\triangle A_k'\right] \leq \sum\limits_{k=1}^r 4^r\left|p_k-p'_k\right| \leq 4^r\|p-p'\|_{\ell^1}.\]
\end{proof}
The technical importance of this result is, of course, that it allows us to obtain continuity in the problem with respect to changes in the distribution constraint on the stopping times.

While we have demonstrated that a lot can be said about distribution-constrained stopping times using the representation in Lemma~\ref{Lem:PartitionCharacterization}, it turns out that we can obtain a more manageable representation if we introduce extra controlled processes which represent the conditional probability of the stopping time taking on each possible value. This vector-valued stochastic process is a martingale in a probability simplex. In the next result, we make clear the connection between this process and a distribution-constrained stopping time.

In the remainder of the paper, we define $\mathcal{A}$ to denote the collection of all progressively-measurable, square-integrable, $\mathbb{R}^r$-valued processes which are independent of $\mathcal{F}_0$. We also denote
\[Y^{t,y,\alpha}_u := y + \int_t^u \alpha_s dW_s\]
for all $(t,y)\in[0,\infty)\times\mathbb{R}^r$, $u\in[t,\infty)$, and $\alpha\in\mathcal{A}$. When needed, we will denote the $k$th coordinate of this vector-valued process by $Y^{(k),t,y,\alpha}$. We will occasionally abuse notation and leave out superscripts when they are clearly implied by the context.

We also denote by $\Delta$ the following closed and convex set
\[\Delta := \left\{y=(y_1,\ldots,y_r)\in[0,1]^r \mid y_1 + \cdots + y_r = 1\right\} \subset \mathbb{R}^r.\]
We then can state a lemma regarding a characterization of distribution-constrained stopping times in terms of a state-constrained controlled martingale.

\begin{lem}\label{Lem:ControlCharacterization}
A stopping time $\tau \in \mathcal{T}$ has the distribution $\mu$ if and only if it is of the form
\begin{equation}\nonumber
\tau = \min\limits_{k\in\{1,\ldots,r\}}\left\{ t_k \mid Y^{(k),0,p,\alpha}_{t_k} = 1\right\},
\end{equation}
almost-surely, for some $\alpha\in\mathcal{A}$ such that
\[Y^{0,p,\alpha}_t \in \Delta,\]
almost-surely, for all $t\geq 0$, and 
\[Y^{(k),0,p,\alpha}_{t_k} \in \{0,1\},\]
almost-surely, for each $k\in\{1,\ldots,r\}$.
\end{lem}

\begin{proof}

\textit{Step 1:} Let $\alpha\in\mathcal{A}$ be a control for which $Y^{0,p,\alpha}_t \in \Delta$, almost-surely, for all $t\geq 0$ and $Y^{(k),0,p,\alpha}_{t_k}\in\{0,1\}$, almost-surely, for each $k\in\{1,\ldots,r\}$. Define $\tau$ as
\[\tau := \min\limits_{k\in\{1,\ldots,r\}}\left\{ t_k \mid Y^{(k),0,p,\alpha}_{t_k} = 1\right\}.\]
It is clear from the properties above that $Y^{(k),0,p,\alpha}_{t_r}\in\{0,1\}$ for every $k\in\{1,\ldots,r\}$ and $Y^{0,p,\alpha}_{t_r} \in \Delta$, which implies that $\tau \leq t_r$, almost-surely. Then $\tau\in\mathcal{T}$, but we must check that it has $\mu$ as its distribution.

Fix $k\in\{1,\ldots,r\}$ and note that
\begin{equation}\nonumber
\mathbb{P}\left[\tau=t_k\right] = \mathbb{P}\left[\underbrace{\{Y^{(1),0,p,\alpha}_{t_1}=0\}\cap\cdots\cap\{Y^{(k-1),0,p,\alpha}_{t_{k-1}}=0\}}_A\cap\underbrace{\{Y^{(k),0,p,\alpha}_{t_k}=1\}}_B\right].
\end{equation}
Note that $B \subset A$ up to a set of measure zero because in the set $B\setminus A$, we have $Y^{(k),0,p,\alpha}_{t_k}=1$ as well as $Y^{(\ell),0,p,\alpha}_{t_\ell}=1$ for some $\ell<k$. Because $Y^{0,p,\alpha}$ is a martingale constrained to $\Delta$, this implies $Y^{(\ell),0,p,\alpha}_{t_k}=1$, almost-surely, which contradicts $Y^{0,p,\alpha}_{t_k}\in\Delta$. Then, we can conclude
\begin{equation}\nonumber
\mathbb{P}\left[\tau=t_k\right] = \mathbb{P}\left[Y^{(k),0,p,\alpha}_{t_k}=1\right]=p_k
\end{equation}
because $Y^{(k),0,p,\alpha}_0=p_k$ and $Y^{(k),0,p,\alpha}_t$ is a martingale taking values zero and one at $t_k$.

\vspace{5mm}\textit{Step 2:} Let $\tau\in\mathcal{T}$ be a stopping time such that $\tau\sim\mu$. Then define the $[0,1]^r$-valued process $\overline{Y}$ as
\begin{equation}\nonumber
\overline{Y}^{(k)}_t := \mathbb{E}\left[ 1_{\{\tau=t_k\}} \mid \mathcal{F}_t\right].
\end{equation}
Note that $\overline{Y}_0 = p$. By the Martingale Representation Theorem, there exists a control $\alpha\in\mathcal{A}$ for which $Y^{0,p,\alpha}_t = \overline{Y}_t$, almost-surely, for all $t\geq 0$. We can then check that,
\begin{equation}\nonumber
Y^{(1),0,p,\alpha}_t+\cdots+Y^{(r),0,p,\alpha}_t = \mathbb{E}\left[1_{\{\tau=t_1\}}+\cdots+1_{\{\tau=t_r\}}\mid\mathcal{F}_t\right] = 1,
\end{equation}
so $Y^{0,p,\alpha}_t\in\Delta$ for all $t\geq 0$, almost-surely. Finally, for any $k\in\{1,\ldots,r\}$, we have $Y^{(k),0,p,\alpha}_{t_k} = 1_{\{\tau=t_k\}}\in\{0,1\}$ because $\{\tau=t_k\}$ is $\mathcal{F}_{t_k}$-measurable.

Define a stopping time $\sigma$ as
\begin{equation}\nonumber
\sigma := \min\limits_{k\in\{1,\ldots,r\}}\left\{t_k\mid Y^{(k),0,p,\alpha}_{t_k}=1\right\}
\end{equation}
and suppose that there exists a set $A$ of non-zero probability on which $\tau\neq\sigma$. Then, for some $k,\ell\in\{1,\ldots,r\}$ such that $k\neq\ell$, the set $B:=A\cap\{\tau=t_k\}\cap\{\sigma=t_\ell\}$ has non-zero probability.

Suppose that $\ell< k$. Then $Y^{(\ell),0,p,\alpha}_{t_\ell}=1$ on $B$ and because $Y^{0,p,\alpha}$ is a martingale constrained to $\Delta$, it follows that $Y^{(\ell),0,p,\alpha}_{t_k}=1$ on $B$, and consequently, $Y^{(k),0,p,\alpha}_{t_k}=1_{\{\tau=t_k\}}=0$, which contradicts $\tau=t_k$ on $B$. On the other hand, suppose that $\ell >k$. Then $Y^{(k),0,p,\alpha}_{t_k}\neq 1$ on $B$, but because $Y^{(k),0,p,\alpha}_{t_k}=1_{\{\tau=t_k\}}$ this also contradicts $\tau=t_k$ on $B$. We conclude $\tau=\sigma$, almost-surely.
\end{proof}

\subsection{Solution via iterated stochastic control}\label{Section:IteratedControl}

It is convenient to define a sequence of sets which will be important in the remainder of the paper. For each $k\in\{1,\ldots,r\}$, define
\[\Delta_k := \{ (y_1,\ldots,y_r) \in \Delta \mid y_\ell = 0\text{ for each }\ell\in\{1,\ldots,k-1\}\} \subseteq \Delta.\]
Note that each set is closed and convex and $\Delta_{k+1} \subset \Delta_k$ for each $k\in\{1,\ldots,r-1\}$.

We also introduce subcollections of stopping times and controls with additional independence properties, which will be used later in proofs of the Dynamic Programming Principle. In particular, for any $t\in[0,\infty)$ we define a subcollection of stopping times
\[\mathcal{T}_t := \left\{ \tau \in \mathcal{T} \mid \tau\text{ is independent of }\mathcal{F}_t\right\} \subset \mathcal{T}\]
and a subcollection of controls
\[\mathcal{A}_t := \left\{ \alpha \in \mathcal{A} \mid \alpha\text{ is independent of }\mathcal{F}_t\right\} \subset \mathcal{A}.\]

We then define a sequence of iterated distribution-constrained optimal stopping problems.

\begin{defn}
For each $k\in\{1,\ldots,r\}$, define a function $v_k : \mathbb{R}\times\Delta_k\to\mathbb{R}$ as
\begin{equation}\label{Eqn:IteratedStoppingRepresentation}
\begin{array}{rccl}
v_k(x,y) & := & \sup\limits_{\tau\in\mathcal{T}_{t_{k-1}}} & \mathbb{E}\left[f(X^{t_{k-1},x}_\tau)\right] \\
& & \text{s.t.} & \tau \sim \sum_{\ell=1}^r y_\ell \delta_{t_\ell}.
\end{array}\end{equation}
\end{defn}
Note that $v^\star = v_1(x_0,p)$. We emphasize that while each $v_k$ is written as a function depending on an entire tuple $y=(y_1,\ldots,y_r)\in\Delta_k$, we have $y_1=\cdots=y_{k-1}=0$ by the definition of $\Delta_k$.

Our goal is to convert these iterated distribution-constrained optimal stopping problems into iterated state-constrained stochastic control problems.

First, we record a growth and continuity estimate for each $v_k$.

\begin{prop}\label{Proposition:LipschitzEstimate}
There exists $C>0$, which depends only on $f$ and $\mu$, for which
\begin{eqnarray}
\left|v_k(x,y)\right| & \leq & C\left(1+|x|\right) \nonumber\\
\left|v_k(x,y) - v_k(x',y')\right| & \leq & C\left(\left|x-x'\right|+\|y-y'\|^{1/2}\right)\nonumber
\end{eqnarray}
for each $k\in\{1,\ldots,r\}$ and all $(x,y),(x',y')\in\mathbb{R}\times\Delta_k$.
\end{prop}

We do not specify the choice of norm in $y$ because it only affects the choice of constant.

\begin{proof}
Recall that $f$ is assumed to be Lipschitz-continuous. Let $\tau\in\mathcal{T}_{t_{k-1}}$ be an arbitrary stopping time such that $\tau\sim\sum_{k=1}^r y_k\delta_{t_k}$ (such a stopping time exists by Corollary \ref{Cor:ExistenceConstruction}). Then, we have
\begin{eqnarray}
\left|\mathbb{E}\left[f(X^{t_{k-1},x}_\tau)\right]\right| & \leq & \mathbb{E}\left[\left|f(X^{t_{k-1},x}_\tau)\right|\right]\nonumber\\
& \leq & \left|f(0)\right| + L\left(|x| + \mathbb{E}\left[\left|W_\tau-W_{t_{k-1}}\right|\right]\right)\nonumber\\
& \leq & \left|f(0)\right| + L\left(|x| + 2\mathbb{E}\left[\left|W_{t_r}\right|\right]\right)\nonumber\\
& \leq & \left|f(0)\right| + L\left(|x| + 2\sqrt{\frac{2}{\pi}t_r}\right).\nonumber
\end{eqnarray}
Next, we compute
\begin{eqnarray}
v_k(x',y) & \geq & \mathbb{E}\left[f(X^{t_{k-1},x'}_\tau)\right] \nonumber\\
& \geq & \mathbb{E}\left[f(X^{t_{k-1},x}_\tau)\right] - L\left|x-x'\right|.\nonumber
\end{eqnarray}
Lastly, by applying Proposition~\ref{Proposition:ApproximateStoppingTimeConstruction}, we can construct a stopping time $\tau'\in\mathcal{T}_{t_{k-1}}$ such that $\tau' \sim \sum_{k=1}^r y_k \delta_{t_k}$ and
\[\mathbb{P}\left[\tau\neq\tau'\right] \leq 4^r \|y-y'\|.\]
Then we can compute
\begin{eqnarray}
v_k(x,y') & \geq & \mathbb{E}\left[f(X^{t_{k-1},x}_{\tau'})\right] \nonumber\\
& \geq & \mathbb{E}\left[f(X^{t_{k-1},x}_\tau)\right] - C\mathbb{E}\left[1_{\tau\neq\tau'}\left|X^{t_{k-1},x}_{\tau'}-X^{t_{k-1},x}_\tau\right|\right] \nonumber\\
& \geq &  \mathbb{E}\left[f(X^{t_{k-1},x}_\tau)\right] - 2C\mathbb{E}\left[1_{\tau\neq\tau'}\left|W_{t_r}\right|\right]\nonumber\\
& \geq & \mathbb{E}\left[f(X^{t_{k-1},x}_\tau)\right] - 2C\,2^{r}\,t_r^{1/2} \|y-y'\|_{\ell^1}^{1/2}.\nonumber
\end{eqnarray}
Then the stated results hold because $(x,y)\in\mathbb{R}\times\Delta_k$ and $\tau$ were both arbitrary.
\end{proof}

In the remainder of the paper, it will prove useful to consider a type of perspective map on the sets $\Delta_k$. For each $k\in\{1,\ldots,r\}$, define $P_k : \Delta_k \to \Delta_k$ as
\begin{equation}\label{Equation:PerspectiveMap}
P_k(y_1,\ldots,y_r) := \left\{\begin{array}{cl}
(y_1,\ldots,y_r) & \text{ if }y_k = 1 \\
(y_{k+1}+\cdots+y_r)^{-1}(0,\ldots,0,y_{k+1},\ldots,y_r) & \text{ if }y_k < 1.
\end{array}\right.
\end{equation}
We note three key properties of this map.
\begin{enumerate}
\item For any $y\in\Delta_k\setminus\{e_k\}$, we have $P_k(y) \in \Delta_{k+1}$,
\item For any $y\in\Delta_k$, the $k$th coordinate of $P_k(y)$ is either zero or one, and
\item The map $P_k$ is continuous on $\Delta_k\setminus\{e_k\}$.
\end{enumerate}

We now provide a dynamic programming lemma whose proof has the same flavour of the weak dynamic programming results in \cite{BouchardTouzi2011,BouchardNutz2012,BayraktarYao2013}. Compared to these previous results, we have a priori continuity of the value functions on the right-hand-side, so we do not need to consider upper- and lower-semicontinuous envelopes. We extend the ideas of a countable covering of the state-space by balls, each associated with a nearly optimal stopping time. To deal with the state-constraints, we employ an argument that utilizes the compactness and convexity of $\Delta_k$ along with the continuity of $v_{k+1}$. The proof of this lemma is the heart of the paper, but is quite involved, so it is relegated to the appendix.

\begin{lem}[Dynamic Programming]\label{Lem:WeakDPPLemma}
For every $k\in\{1,\ldots,r-1\}$ and every $(x,y)\in\mathbb{R}\times\Delta_k$, we have
\begin{equation}\label{Eqn:TerminallyConstrainedControlRepresentation}
\begin{array}{rccl}
v_k(x,y) & = & \sup\limits_{\alpha\in\mathcal{A}_{t_{k-1}}} & \mathbb{E}\left[Y^{(k),t_{k-1},y,\alpha}_{t_k} f(X^{t_{k-1},x}_{t_k})+(1-Y^{(k),t_{k-1},y,\alpha}_{t_k} )v_{k+1}\left(X^{t_{k-1},,x}_{t_k},Y^{t_{k-1},y,\alpha}_{t_k}\right)\right] \\
& & \text{s.t.} & Y^{t_{k-1},y,\alpha}_u \in \Delta_k\text{ for all }u\geq t_{k-1} \\
& & & Y^{(k),t_{k-1},y,\alpha}_{t_k} \in \{0,1\}\text{, almost-surely.}
\end{array}\end{equation}
\end{lem}

\begin{proof}

See Appendix \ref{Appendix:WeakDPPLemma}.

\end{proof}

Next, we demonstrate that we may relax the terminal constraint. The proof of this idea relies on a careful construction of a perturbed martingale which satisfies the terminal constraints of the previous problem, but does not significantly change the expected payoff. The proof of this result shares many key ideas with the previous lemma. For the sake of exposition, we provide this proof in the appendix as well.

\begin{lem}[Constraint Relaxation]\label{Lem:TerminalRelaxation}
For every $k\in\{1,\ldots,r-1\}$ and every $(x,y)\in\mathbb{R}\times\Delta_k$, we have
\begin{equation}\label{Eqn:StateConstrainedControlRepresentation}
\begin{array}{rccl}
v_k(x,y) & = & \sup\limits_{\alpha\in\mathcal{A}_{t_{k-1}}} & \mathbb{E}\left[Y^{(k),t_{k-1},y,\alpha}_{t_k} f(X^{t_{k-1},x}_{t_k})+(1-Y^{(k),t_{k-1},y,\alpha}_{t_k} )v_{k+1}\left(X^{t_{k-1},x}_{t_k},P_k(Y^{t_{k-1},y,\alpha}_{t_k})\right)\right] \\
& & \text{s.t.} & Y^{t_{k-1},y,\alpha}_u \in \Delta_k\text{ for all }u\geq t_{k-1}\text{, almost-surely,}
\end{array}\end{equation}
where $P_k : \Delta_k \to \Delta_k$ is the perspective map defined in \eqref{Equation:PerspectiveMap}.
\end{lem}

Note, even though $P_k(e_k)\not\in\Delta_{k+1}$, the right-hand-side of \eqref{Eqn:StateConstrainedControlRepresentation} is well-defined because $v_{k+1}$ is known to be bounded and continuous. Then, there is a unique continuous extension of the map $(x,y) \mapsto (1-y_k)v_{k+1}(x,y)$ from $\Delta_k\setminus\{e_k\}$ to $\Delta_k$. That is, taking the right-hand-side to be zero when $y=e_k$.

\begin{proof}

See Appendix \ref{Appendix:ProofTerminalRelaxation}.

\end{proof}

%

%
%

With these lemmas in hand, we can now state the main result of this paper.

\begin{thm}\label{Thm:MainResult}
The function $v_r : \mathbb{R}\times\Delta_r \to \mathbb{R}$ satisfies
\begin{equation}\nonumber
v_r(x,y) = \mathbb{E}\left[f(X^{t_{r-1},x}_{t_r})\right]
\end{equation}
for every $(x,y)\in\mathbb{R}\times\Delta_r$.

For each $k\in\{1,\ldots,r-1\}$, the function $v_k : \mathbb{R}\times\Delta_k \to \mathbb{R}$ is the value function of the following state-constrained stochastic control problem
\begin{equation}\nonumber\begin{array}{rccl}
v_k(x,y) & = & \sup\limits_{\alpha\in\mathcal{A}_{t_{k-1}}} & \mathbb{E}\left[Y^{(k),t_{k-1},y,\alpha}_{t_k} f(X^{t_{k-1},x}_{t_k})+(1-Y^{(k),t_{k-1},y,\alpha}_{t_k} )v_{k+1}\left(X^{t_{k-1},x}_{t_k},P_k(Y^{t_{k-1},y,\alpha}_{t_k})\right)\right] \\
& & \text{s.t.} & Y^{t_{k-1},y,\alpha}_u \in \Delta_k\text{ for all }u\geq t_{k-1}\text{, almost-surely,}
\end{array}\end{equation}
where $P_k : \Delta_k \to \Delta_k$ is defined as in \eqref{Equation:PerspectiveMap}.

Of course, we then have
\begin{equation}\nonumber
v^\star = v_1(x_0,p_1,\ldots,p_r).
\end{equation}
\end{thm}

\begin{proof}

It is clear that $v_r$ has the representation above because there is only one admissible stopping time. 
The rest follows immediately from applying Lemma~\ref{Lem:WeakDPPLemma} and Lemma~\ref{Lem:TerminalRelaxation}.

\end{proof}

\subsection{Time-dependent value functions}\label{Subsection:TimeDependentValueFunctions}

While, for the purposes of this paper, we may consider the results of Theorem~\ref{Thm:MainResult} as a solution to the distribution-constrained optimal stopping problem, we can consider an additional time-dependent version of the state-constrained problem which is amenable to numerical resolution. In particular, the time-dependent value functions will correspond to viscosity solutions of Hamilton-Jacobi-Bellman (HJB) equations.

\begin{defn}\label{Definition:TimeDependentValue}
Define a function $w_r : [t_{r-1},t_r]\times\mathbb{R}\times\Delta_r\to\mathbb{R}$ as
\begin{equation}\nonumber
w_r(t,x,y) := \mathbb{E}\left[f(X^{t,x}_{t_r})\right].
\end{equation}
For each $k\in\{1,\ldots,r-1\}$, define a function $w_k : [t_{k-1},t_k]\times\mathbb{R}\times\Delta_k\to\mathbb{R}$ as
\begin{equation}\nonumber\begin{array}{rccl}
w_k(t,x,y) & := & \sup\limits_{\alpha\in\mathcal{A}} & \mathbb{E}\left[Y^{(k),t,y,\alpha}_{t_k} f(X^{t,x}_{t_k})+(1-Y^{(k),t,y,\alpha}_{t_k})v_{k+1}\left(X^{t,x}_{t_k},P_k(Y^{t,y,\alpha}_{t_k})\right)\right] \\
& & \text{s.t.} & Y^{t,y,\alpha}_u \in \Delta_k\text{ for all }u\geq t\text{, almost-surely,}
\end{array}\end{equation}
where $P_k : \Delta_k \to \Delta_k$ is defined as in \eqref{Equation:PerspectiveMap}.
\end{defn}

\begin{remark}\label{Remark:EquivalentAdmissibleSet}
Using the properties of the time-independent auxiliary value functions $v_k$ deduced in Section~\ref{Section:IteratedControl}, we can consider the terminal payoff as a given H\"older continuous function. Then this is a true state-constrained optimal stochastic control problem. In particular, we note that by the argument of Remark~5.2 in \cite{BouchardTouzi2011} we can equivalently define $w_k$ as a supremum over controls in $\mathcal{A}$ which are not necessarily independent of $\mathcal{F}_t$.
\end{remark}

We note an immediate relationship with the value functions of Section \ref{Section:IteratedControl}.

\begin{cor}\label{Corollary:TimeDependentRelationship}
For each $k\in\{1,\ldots,r\}$ we have
\[v_k\left(x,y\right) = w_k(t_{k-1},x,y)\]
for all $(x,y)\in\mathbb{R}\times\Delta_k$.
\end{cor}

\begin{proof}
This result is obvious from the definition of $w_k$ and Theorem~\ref{Thm:MainResult}.
\end{proof}

Before stating a Dynamic Programming Principle for the time-dependent value functions, we first investigate their regularity. In particular, we aim to demonstrate that each $w_k$ is concave in $y$ and jointly continuous with estimates on its modulus of continuity. 

\begin{prop}\label{Proposition:Concavity}
For each $k\in\{1,\ldots,r\}$, the function $w_k(t,x,\cdot)$ is concave for each $(t,x)$.
\end{prop}

\begin{proof}

We proceed by backwards induction. Notice that the set $\Delta_r$ is a singleton, so the functions $w_r$ and $v_r$ are both trivially concave in $y$.

Suppose that $v_{k+1}$ is concave in $y$ for some $k\in\{1,\ldots,r-1\}$. The key observation is that the map
\begin{eqnarray}
\Delta_k\setminus\{e_k\} \ni y & \mapsto & (1-y_k)v_{k+1}\left(x,P_k(y)\right) \nonumber\\
& = & (y_{k+1}+\cdots+y_r)v_{k+1}\left(x,\frac{(0,\ldots,0,y_{k+1},\ldots,y_r)}{y_{k+1}+\cdots+y_r}\right)\nonumber
\end{eqnarray}
is concave for every $x\in\mathbb{R}$ because it is the perspective transformation of the concave map $\Delta_{k+1} \ni y\mapsto v_{k+1}(x,y)$ (See Section~3.2.6 in \cite{BoydVandenberghe2004}).

With this in mind fix any $(t,x)\in[t_{k-1},t_k]\times\mathbb{R}$, $y_1,y_2\in\Delta_k$, and $\lambda\in[0,1]$. Let $\alpha_1,\alpha_2\in\mathcal{A}$ be arbitrary controls for which
\[Y^{t,y_1,\alpha_1}_u,Y^{t,y_2,\alpha_2}_u \in \Delta_k,\]
almost-surely, for all $u\geq t$. Define $\overline{y} := \lambda y_1 + (1-\lambda)y_2$ and $\overline\alpha_u := \lambda\alpha_{1,u}+(1-\lambda)\alpha_{2,u}$. Then $\overline\alpha \in \mathcal{A}$ and
\[Y^{t,\overline{y},\overline\alpha}_u \in \Delta_k,\]
almost-surely, for all $u\geq t$ by the convexity of the set $\Delta_k$.

Then using the concavity of the perspective map, we can compute
\begin{eqnarray}
w_k(t,x,\overline{y}) & \geq & \mathbb{E}\left[Y^{(k),t,\overline{y},\overline\alpha}_{t_k} f(X^{t,x}_{t_k})+(1-Y^{(k),t,\overline{y},\overline\alpha}_{t_k})v_{k+1}\left(X^{t,x}_{t_k},P_k(Y^{t,\overline{y},\overline\alpha}_{t_k})\right)\right]\nonumber\nonumber\\
& \geq & \mathbb{E}\left[Y^{(k),t,\overline{y},\overline\alpha}_{t_k} f(X^{t,x}_{t_k})+\lambda(1-Y^{(k),t,y_1,\alpha_1}_{t_k})v_{k+1}\left(X^{t,x}_{t_k},P_k(Y^{t,y_1,\alpha_1}_{t_k})\right)\right.\nonumber\\
& & \hspace{1cm}\left. +(1-\lambda)(1-Y^{(k),t,y_2,\alpha_2}_{t_k})v_{k+1}\left(X^{t,x}_{t_k},P_k(Y^{t,y_2,\alpha_2}_{t_k})\right)\right]\nonumber\\
& = & \lambda\,\mathbb{E}\left[Y^{(k),t,y_1,\alpha_1}_{t_k} f(X^{t,x}_{t_k})+(1-Y^{(k),t,y_1,\alpha_1}_{t_k})v_{k+1}\left(X^{t,x}_{t_k},P_k(Y^{t,y_1,\alpha_1}_{t_k})\right)\right]\nonumber\\
& & \hspace{1cm} +(1-\lambda)\,\mathbb{E}\left[Y^{(k),t,y_2,\alpha_2}_{t_k} f(X^{t,x}_{t_k})+(1-Y^{(k),t,y_2,\alpha_2}_{t_k})v_{k+1}\left(X^{t,x}_{t_k},P_k(Y^{t,y_2,\alpha_2}_{t_k})\right)\right].\nonumber
\end{eqnarray}
But because $\alpha_1,\alpha_2$ were arbitrary, we conclude
\[w_k(t,x,\overline{y}) \geq \lambda w_k(t,x,y_1) + (1-\lambda)w_k(t,x,y_2).\]
Then $w_k$ is concave in $y$, and hence so is $v_k$ by Corollary~\ref{Corollary:TimeDependentRelationship}. Then the result holds by induction.

\end{proof}

We can go a step further and obtain a detailed estimate of the joint continuity of $w_k$.

\begin{prop}\label{Proposition:TimeDependentHolderEstimates}
There exists $C>0$, which depend only on $f$ and $\mu$, such that for each $k\in\{1,\ldots,r-1\}$, we have
\[\left|w_k(t,x,y)-w_k(t',x',y')\right| \leq C\left(\left|t-t'\right|^{1/4} + \left|x-x'\right| + (1+\left|x\right|+\left|x'\right|)\|y-y'\|_{\ell^2}^{1/4})\right)\]
for all $(t,x,y),(t',x',y')\in[t_{k-1},t_k]\times\mathbb{R}\times\Delta_k$.
\end{prop}

The proof of this statement is relatively straightforward but long-winded, so we relegate it to the appendix. We do not claim that these H\"older exponents are sharp.

\begin{proof}

See Appendix~\ref{Appendix:TimeDependentHolderEstimates}.

\end{proof}

%
%

%

The upside of this representation as  an optimal stochastic control problem is that we can characterize each time-dependent value function $w_k$ as a viscosity solution of a corresponding HJB equation. At this point, we can prove a Dynamic Programming Principle for the time-dependent value functions. While these are state-constrained stochastic control problems, we can directly use the a priori continuity of $w_k$ in $y$ and convexity of $\Delta_k$ as in the proof of Lemma~\ref{Lem:WeakDPPLemma}.

\begin{thm}\label{Thm:TimeDependentDPP}
Fix $k\in\{1,\ldots,r-1\}$, $(t,x,y)\in[t_{k-1},t_k)\times\mathbb{R}\times\Delta_k$, and any $h>0$ such that $t+h<t_k$. Let $\{\tau^\alpha\}_{\alpha\in\mathcal{A}_t}$ be a family of stopping times independent of $\mathcal{F}_t$ and valued in $[t,t+h]$. Then
\[\begin{array}{ccl}
w_k(t,x,y) = & \sup\limits_{\alpha\in\mathcal{A}_t} & \mathbb{E}\left[w_k(\tau^\alpha, X^{t,x}_{\tau^\alpha}, Y^{t,y,\alpha}_{\tau^\alpha})\right]\\
&\text{s.t.}& Y^{t,y,\alpha}_u \in \Delta_k\text{ for all }u\geq t\text{, almost-surely.}
\end{array}\]
\end{thm}

\begin{proof}

See Appendix~\ref{Appendix:TimeDependentDPP}.

\end{proof}

From this result, we  can immediately verify that each time-dependent value function is a viscosity solution of an HJB. Once we have the Dynamic Programming Principle in hand, this result becomes reasonably standard, so we direct the interested reader to \cite{Katsoulakis1994,BouchardNutz2012,Rokhlin2014}.
%

We first define elliptic operators $F_k,G_k : \mathbb{S}^{1+r}\times\Delta_k \to \mathbb{R}$ as
\[\begin{array}{rcl}
F_k(R,y) & := & \sup\left\{\left(\begin{array}{c}1 \\ a\end{array}\right)^\top R \left(\begin{array}{c}1 \\ a\end{array}\right) \mid a \in \mathbb{A}_k(y)\right\}\\
G_k(R,y) & := & \sup\left\{\left(\begin{array}{c}0 \\ a\end{array}\right)^\top R \left(\begin{array}{c}0 \\ a\end{array}\right) \mid a \in \mathbb{A}_k(y),\,\|a\|_{\ell^2}=1\right\},
\end{array}\]
where
\[\mathbb{A}_k(y) := \left\{a\in\mathbb{R}^r\mid \exists\epsilon >0\text{ s.t. }y + a\,(-\epsilon,\epsilon) \subset \Delta_k\right\}.\]
The intuition behind these definitions is that $\mathbb{A}_k(y)$ encodes admissible directions in which a state-constrained martingale starting from $y\in\Delta_k$ may evolve. In particular, a martingale constrained to lie in $\Delta_k$ cannot have non-zero quadratic variation in the outer normal direction on the boundary. The elliptic operator $F_k$ shows up naturally from applying Dynamic Programming, while $G_k$ encodes the concavity in $y$.

The main properties of $F_k$ and $G_k$ are that
\[\begin{array}{rcl}
F_k(R,y) < +\infty & \implies & G_k(R,y)\leq 0\\
G_k(R,y) < 0 & \implies & F_k(R,y) < +\infty.
\end{array}\]
Then following the arguments of \cite{BayraktarSirbu2013}, we can deduce that the value function is a viscosity solution of an equation involving an envelope with $G_k$.


\begin{prop}
The function $w_r$ is the unique solution of the heat equation (in reversed time),
\begin{equation}\nonumber\left\{\begin{array}{ll}
u_t + \frac{1}{2}u_{xx} = 0 & \text{in }[t_{r-1},t_r)\times\mathbb{R}\times\Delta_r \\
u = f & \text{on }\{t=t_r\}\times\mathbb{R}\times\Delta_r.
\end{array}\right.\end{equation}
For each $k\in\{1,\ldots,r-1\}$, $w_k$ is a continuous viscosity solution of the HJB equation,
\begin{equation}\label{Equation:TimeDependentPDE}\left\{\begin{array}{rl}
\min\left\{u_t + F_k(D^2_{xy}u, y),\,G_k(D^2_{xy}u,y)\right\} = 0 & \text{in }[t_{k-1},t_k)\times\mathbb{R}\times\Delta_k \\
u = y_k f(x) + (1-y_k)v_{k+1}(x,P_k(y)) & \text{on }\{t=t_k\}\times\mathbb{R}\times \Delta_k.
\end{array}\right.\end{equation}
\end{prop}

The proof of this statement follows from a standard argument and an additional analysis of admissible controls on the boundaries. For more details on the introduction of the operator $G_k$ to obtain a variational inequality, we refer the interested reader to \cite{BayraktarSirbu2013} or Section~4 in \cite{Pham2009}.

The more important question, of course, is whether or not one can obtain a uniqueness result for viscosity solutions of \eqref{Equation:TimeDependentPDE}. It is standard to show that \eqref{Equation:TimeDependentPDE} admits a comparison principle when we have Dirichlet conditions on the boundary of $\Delta_k$. In the following, we demonstrate that we can prove uniqueness even with second-order boundary conditions using the special structure of the domain.

\begin{thm}\label{Thm:Uniqueness}
There is a unique continuous viscosity solution of \eqref{Equation:TimeDependentPDE} which satisfies
\[\left|u(t,x,y)\right| \leq C\left(1+\left|x\right|\right)\]
for some $C>0$.
\end{thm}
Of course the time-dependent value functions satisfy this linear growth constraint as a corollary of Proposition~\ref{Proposition:TimeDependentHolderEstimates}.

The key idea in this proof is that when we restrict a viscosity solution of \eqref{Equation:TimeDependentPDE} to the relative interior of any face $F$ of $\Delta_k$ in the $y$-coordinate, the restricted function is a viscosity solution of the same equation on a smaller state-space. In particular, when restricted to a vertex the equation reduces to the heat equation, for which we immediately have uniqueness. We then apply the comparison principle corresponding to the equation restricted to an edge using the fact that we have uniqueness on the vertices to deduce uniqueness on edges. We proceed as such on higher dimensional faces until we prove uniqueness on all of $\Delta_k$.

\begin{proof}[Sketch of Proof]

Fix $k\in\{1,\ldots,r-1\}$ and let $u,v : [t_{k-1},t_k]\times\mathbb{R}\times\Delta_k \to \mathbb{R}$ be two continuous viscosity solutions of \eqref{Equation:TimeDependentPDE}. Suppose that $u > v$ at some point. By the terminal condition, this must occur at some $t_0 < t_k$.

Suppose that there exists a vertex $y_0$ of the simplex $\Delta_k$ such that $u\neq v$ at some point when restricting to $y_0$ in the $y$-coordinate. By Proposition~6.9 in \cite{Touzi2013}, we note that both $u$ and $v$ are viscosity solutions of the heat equation when restricting the the vertex $y_0$ in the $y$-coordinate. But this contradicts uniqueness for the heat equation.

Let $F$ be a minimal dimension face of the simplex $\Delta_k$ such that $u\neq v$ at some point when restricting to $F$ in the $y$-coordinate. Again, by Propositon~6.9 in \cite{Touzi2013}, we conclude that both $u$ and $v$ are viscosity solutions of the same equation on the relative interior of $F$. Of course, the boundary of $F$ is the union of lower dimension faces of the simplex $\Delta_k$, so, by the assumed minimal dimension property of $F$, we conclude that $u=v$ when restricting to the boundary of $F$ in the $y$-coordinate. Then by applying the comparison principle for \eqref{Equation:TimeDependentPDE} with Dirichlet boundaries, we deduce that $u=v$ on $F$.

The theorem then follows by considering the $k$-dimensional face, $\Delta_k$ itself.

\end{proof}


\section{Application to Superhedging with a Volatility Outlook}\label{Section:Application}

In this section, we consider an application of distribution-constrained optimal stopping in mathematical finance. In particular, we consider the problem of model-free superhedging of a contingent claim with payoff $f(X_T)$ using only dynamic trading in an underlying asset $X$.

We assume that the price process $X_t$ is a martingale under some unknown martingale measure $\mathbb{Q}$, but do not specify the exact volatility dynamics. However, in this problem, we assume that we have an outlook on the volatility in the form of the distribution of the quadratic variation, $\langle X\rangle_T$.\footnote{We note that, while it may seem unlikely that we have an atomic measure representing our volatility outlook, this is a reasonable starting place for two reasons. It is possible to approximate more general measures by atomic measures because it is possible to prove continuity of the value function in the Wasserstein topology (See Lemma~3.1 in \cite{CoxKallblad2015}). Second, pricing by allowing only a finite number of scenarios, as opposed to specifying a full continuous-valued model, is standard in industry (e.g. the specification of rates, default, and prepayment scenarios in standard models for securitized products).}

\subsection{Model-free superhedging}

We follow the model-free setting of \cite{GalichonLabordereTouzi2014,BonnansTan2013}. Let $\Omega := \left\{\omega\in C([0,T],\mathbb{R})\mid\omega_0=0\right\}$ be the canonical space equipped with uniform norm $\|\omega\|_\infty := \sup\limits_{0\leq t\leq T}|\omega_t|$, $B$ the canonical process, $\mathbb{Q}_0$ the Wiener measure, $\mathbb{F}:=\{\mathcal{F}_t\}_{0\leq t\leq T}$ the filtration generated by $B$, and $\mathbb{F}^+:=\{\mathcal{F}_t^+\}_{0\leq t\leq T}$ the right-limit of $\mathbb{F}$.

Fix some initial value $x_0\in\mathbb{R}$. Then, we denote
\begin{equation}\nonumber
X_t := x_0 + B_t.
\end{equation}
For any real-valued, $\mathbb{F}$-progressively measurable process $\alpha$ satisfying $\int_0^T \alpha_s^2\,ds < \infty$, $\mathbb{Q}_0$-a.s., we define the probability measure on $(\Omega,\mathcal{F})$,
\begin{equation}\nonumber
\mathbb{Q}^\alpha := \mathbb{Q}_0\circ\left(X^\alpha\right)^{-1},
\end{equation}
where
\begin{equation}\nonumber
X^\alpha_t := x_0 + \int_0^t \alpha_r\,dB_r.
\end{equation}
Then $X^\alpha$ is a $\mathbb{Q}^\alpha$-local martingale. We denote by $\mathcal{Q}$ the collection of all such probability measures $\mathbb{Q}$ on $(\Omega,\mathcal{F})$ under which $X$ is a $\mathbb{Q}$-uniformly integrable martingale. The quadratic variation process $\langle X\rangle=\langle B\rangle$ is universally defined under any $\mathbb{Q}\in\mathcal{Q}$, and takes values in the set of all non-decreasing continuous functions from $\mathbb{R}_+$ to $\mathbb{R}_+$.

Let $\mu$ be a given probability distribution of the form (\ref{Eqn:AtomicMeasure}). Then we consider the problem
\begin{equation}\nonumber\begin{array}{rcl}
\overline{U} := &\sup\limits_{\mathbb{Q}\in\mathcal{Q}}& \mathbb{E}^\mathbb{Q}\left[f(X_T)\right]\\
&\text{s.t.} & \langle X\rangle_T\sim\mu,
\end{array}\end{equation}
where $\mathcal{Q}$ is a collection of admissible martingale measures. This corresponds to a model-free superhedging price in a sense made clear by the duality results in, for example, \cite{BonnansTan2013}.

\subsection{Equivalence to distribution-constrained optimal stopping}

We show that this problem is equivalent to distribution-constrained optimal stopping of Brownian motion.

\begin{prop}
We have
\begin{equation}\nonumber\begin{array}{rclccl}
\overline{U} := & \sup\limits_{\mathbb{Q}\in\mathcal{Q}}& \mathbb{E}^\mathbb{Q}\left[f(X_T)\right] & = & \sup\limits_{\tau\in\mathcal{T}} & \mathbb{E}^{\mathbb{Q}_0}\left[f(X_\tau)\right].\\
&\text{s.t.} & \langle X\rangle_T\sim\mu & & \text{s.t.} & \tau \sim \mu.
\end{array}\end{equation}
\end{prop}

\begin{proof}

This argument can be found in Theorem 2.4 of \cite{BonnansTan2013}. For completeness, we reproduce it below.

Let $\mathbb{Q}\in\mathcal{Q}$ such that the $\mathbb{Q}$-distribution of $\langle X\rangle_T$ is $\mu$. It follows by the Dambis-Dubins-Schwarz Theorem that $X_T=x+\tilde{W}_{\langle X\rangle_T}$ where $\tilde{W}$ is a standard Brownian motion and $\tau := \langle X\rangle_T$ is a stopping time with respect to the time-changed filtration with distribution $\mu$ (See Theorem~4.6 in \cite{KaratzasShreve1991}). Then $\overline{U} \leq \sup\limits_{\tau\sim\mu}\mathbb{E}^{\mathbb{Q}_0}\left[f(X_\tau)\right]$.

Let $\tau$ be a stopping time such that $\tau\sim\mu$. Define a process $X^\tau$ as
\begin{equation}\nonumber
X^\tau_t := x + B_{\tau\wedge\frac{t}{T-t}}.
\end{equation}
Note that $X^\tau$ is a continuous martingale on $[0,T]$ with $\langle X^\tau\rangle_T=\tau$, so $X^\tau$ induces a probability measure $\mathbb{Q}\in\mathcal{Q}$ such that $\langle X^\tau\rangle_T=\tau\sim\mu$. Then the opposite inequality holds.

\end{proof}

Then one can obtain a model-free super-hedging price with a volatility outlook by solving the iterated stochastic control problem in Section \ref{Section:IteratedControl}.

\subsection{Numerical example}\label{Section:Numerics}

In this section we obtain approximate numerical solutions of the distribution-constrained optimal stopping problem using finite-difference schemes.

In particular, we consider two potential outlooks on volatility. In the first, the binary outlook, we assume equal probability between a high- and low-volatility scenario
\begin{equation}
\mu_2 := \frac{1}{2}\delta_{10} + \frac{1}{2}\delta_{20}.\nonumber
\end{equation}
In the second, we augment the binary outlook with a third extreme volatility scenario which occurs with small probability
\begin{equation}
\mu_3 := \frac{9}{20}\delta_{10} + \frac{9}{20}\delta_{20} + \frac{1}{10}\delta_{100}.\nonumber
\end{equation}
Our goal is to compute the model-free superhedging price of a European call option under each volatility outlook. Because we do not restrict to models where the price process is non-negative, we can take the payoff to be $f(x) := x^+$ without loss of generality.

Then, as before, we define value functions for each outlook as
\begin{equation}
v_2(x) := \sup\limits_{\tau\in\mathcal{T}(\mu_2)}\mathbb{E}^x\left[f(W_\tau)\right]\text{ and }v_3(x) := \sup\limits_{\tau\in\mathcal{T}(\mu_3)}\mathbb{E}^x\left[f(W_\tau)\right].\nonumber
\end{equation}

We solve the problem using the iterated stochastic control approach from Section \ref{Section:IteratedControl}. In particular, we obtain a viscosity solution of the corresponding Hamilton-Jacobi-Bellman equation in Section \ref{Subsection:TimeDependentValueFunctions} using a finite-difference scheme. It is important to emphasize that, because of potential degeneracy due to the extra state-variables in $w_2$ and $w_3$, it is critical to use a monotone numerical scheme.

In these results, we apply a version of the wide-stencil scheme introduced in \cite{Oberman2007}. In particular, we approximate the non-linear terms in each equation by monotone finite-difference approximations of the following form
\begin{equation}\nonumber
\sup\limits_{\alpha\in\mathbb{R}}\left[\left(\begin{array}{c}1\\\alpha\end{array}\right)^\top\left(\begin{array}{cc}u_{xx}&u_{xy}\\u_{xy}&u_{yy}\end{array}\right)\left(\begin{array}{c}1\\\alpha\end{array}\right)\right] \approx \max\limits_{k\in\mathcal{K}(t,x,y)} \frac{u(x+h,t,y+k)-2u(x,t,y)+u(x-h,t,y-k)}{h^2},
\end{equation}
where the set $\mathcal{K}(t,x,y)$ is a collection such that $y\pm k$ lies on nearby grid-points. For a rigorous analysis of wide-stencil schemes for degenerate elliptic equations, we refer the reader to \cite{Oberman2008,FroeseOberman2011,Oberman2007}.

For comparison, we consider two main special cases, which we refer to as the ``mean volatility'' value and the ``support-constrained'' value. We define the mean volatility value as the model-free superhedging price obtained by assuming the quadratic variation will be equal to the mean of the distribution in the corresponding distribution-constrained problem. We define their corresponding value functions as $\underline{v}_2$ and $\underline{v}_3$, respectively. On the other hand, we define the support-constrained value as the model-free superhedging price obtained when only restricting the quadratic variation to have the same support as that of the distribution in the corresponding distribution-constrained problem. We define their corresponding value functions as $\overline{v}_2$ and $\overline{v}_3$, respectively.

We expect the following ordering:
\begin{equation}\nonumber
f(x) \leq \underline{v}_2(x) \leq v_2(x) \leq \overline{v}_2(x)
\end{equation}
and
\begin{equation}\nonumber
f(x) \leq \underline{v}_3(x) \leq v_3(x) \leq \overline{v}_3(x).
\end{equation}
Furthermore, we note that we can compute $\underline{v}_2$, $\overline{v}_2$, $\underline{v}_3$, and $\overline{v}_3$ explicitly in terms of heat kernels (See Section~2.3 in \cite{Evans2010}).

\begin{figure}[t]
\centering
\includegraphics[scale=1]{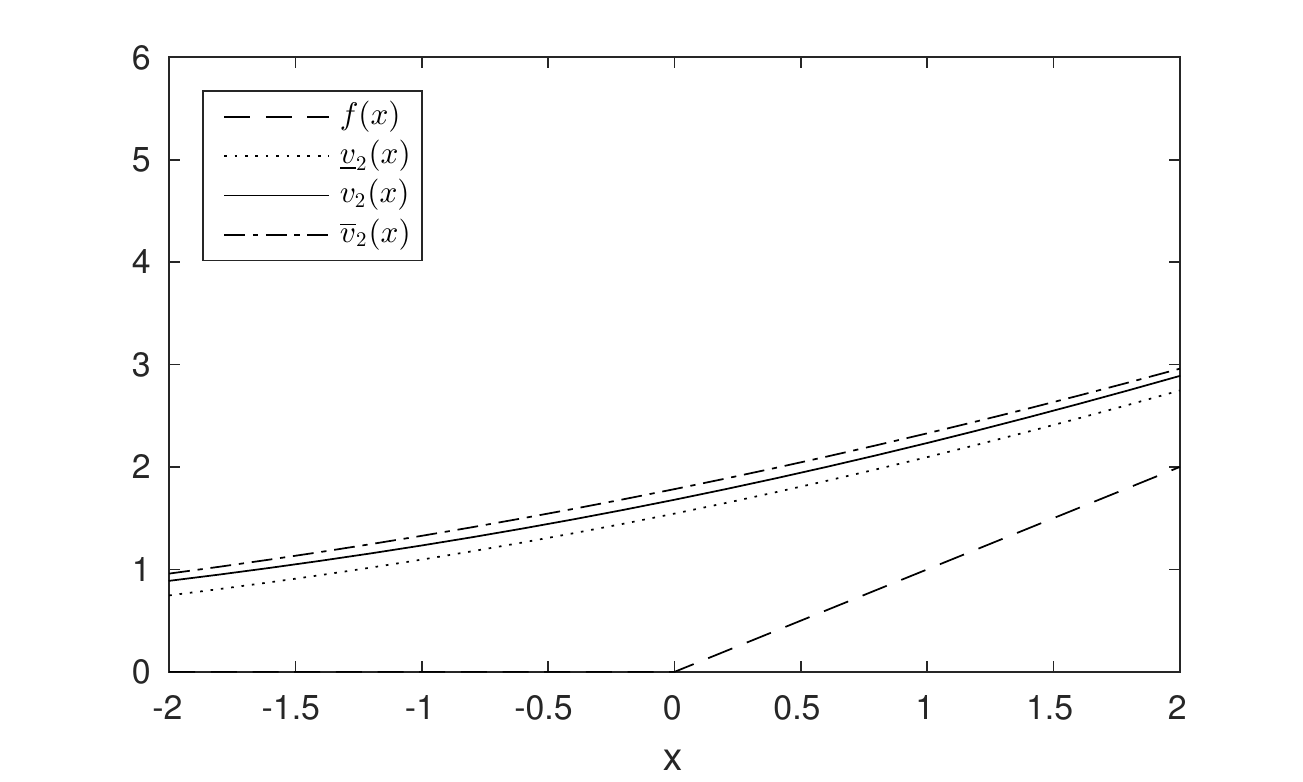}
\caption{Comparison of the model-free superhedging values corresponding to distribution constraints on quadratic variation ($v_2$), support constraints on quadratic variation ($\overline{v}_2$), and averaged quadratic variation ($\underline{v}_2$). Each of these is in the two-atom (binary) volatility outlook. The distribution-constrained value corresponds with the value function of an optimal stopping problem under a two-atom distribution constraint.}\label{Fig:Numerics1}
\end{figure}

\begin{figure}[t]
\centering
\includegraphics[scale=1]{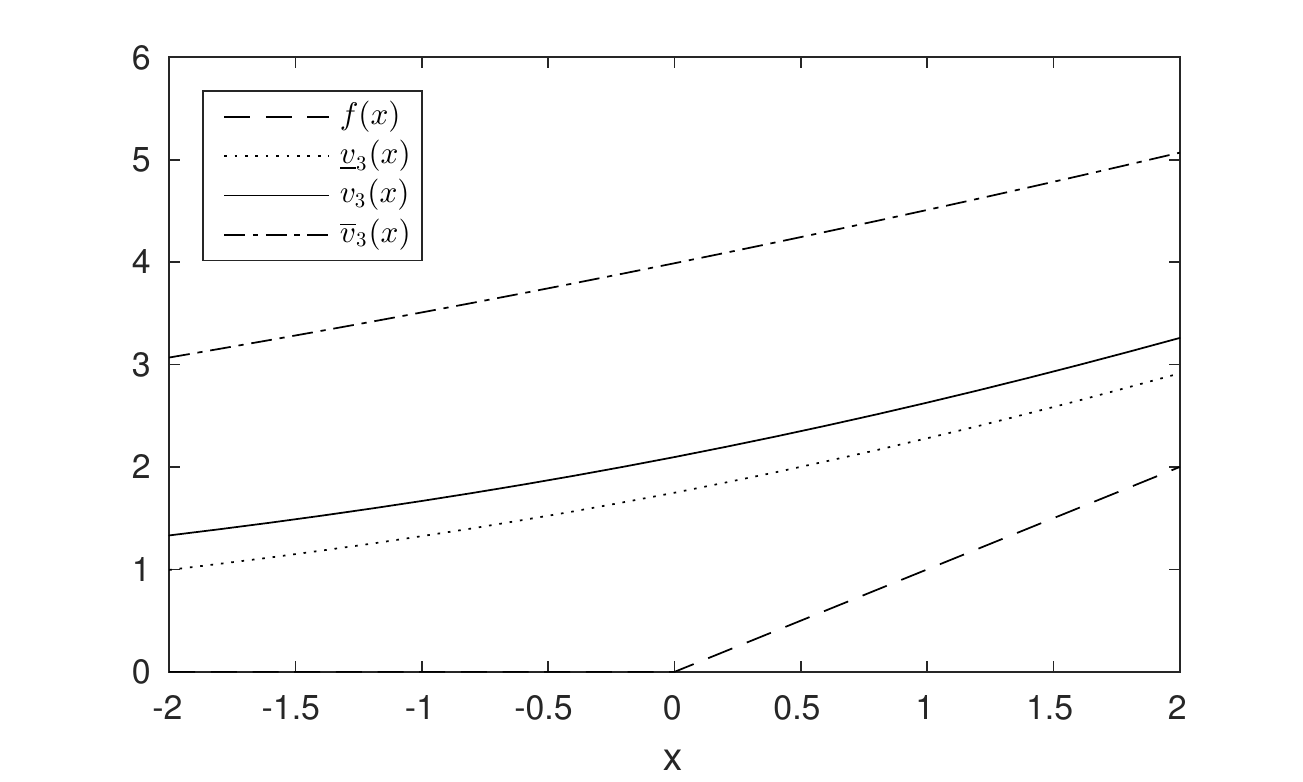}
\caption{Comparison of the model-free superhedging values corresponding to distribution constraints on quadratic variation ($v_3$), support constraints on quadratic variation ($\overline{v}_3$), and averaged quadratic variation ($\underline{v}_3$). Each of these is in the three-atom (trinary) volatility outlook. The distribution-constrained value corresponds with the value function of an optimal stopping problem under a three-atom distribution constraint.}\label{Fig:Numerics2}
\end{figure}

We illustrate the value function for the two- and three-atom problem in Figure \ref{Fig:Numerics1} and Figure \ref{Fig:Numerics2}, respectively. As expected, we see a superhedging value which is increasing in the underlying asset price (or, equivalently, decreasing in the strike price) and respects the bounds implied by the support-constrained and average-volatility models. As expected, the bound provided by the support-constrained superhedging problem is particularly poor in the three-model volatility outlook, where we stipulate that the high volatility (high value) case is rare.

It is interesting to note that careful comparison of the two figures illustrates an increase in superhedging value between the two volatility outlooks which is roughly proportional to the increase in square-root of expected quadratic variation. For example, there is approximately a 25\% increase in value at $x=0$, which is essentially exactly in-line with the 25.2\% increase in square-root of expected quadratic variation between the two outlooks. This matches our intuition that call option superhedging prices should be proportional to expected volatility to first order.

\begin{figure}[t]
\centering
\includegraphics[scale=0.5]{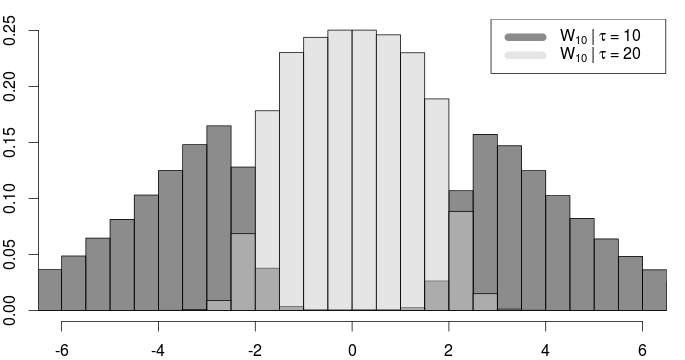}
\caption{Probability density estimates of $W_{10}$ conditional on $\tau=10$ and $\tau=20$ for an optimal stopping time for the two-atom volatility outlook model starting from $W_0=0$. Density estimates were made by Monte Carlo simulations on high-resolution solutions to the associated HJB equations. Sample size, $N=10^7$.}\label{Fig:Numerics3}
\end{figure}

In Figure \ref{Fig:Numerics3}, we provide a probability density estimate of $W_{10}$ conditional on $\tau=10$ and $\tau=20$ for an approximate optimal stopping time for the two-atom volatility outlook model starting from $W_0=0$. We obtain these estimates by performing Monte Carlo simulations with controls estimated from a numerical solution of the associated HJB equations. We use grid spacings $dx=0.1$, $dy=0.005$, and $dt=0.01$. We perform $10^7$ simulations and verify that relevant statistics from the Monte Carlo simulation match those from the finite-difference solutions (e.g., expected payoff, distribution and moments of the stopping time and stopped process) to within a reasonable margin of error.

The density estimates provide insight into the form of the optimal strategy. Recall that the payoff is locally-affine at all points except $x=0$, where it is strictly convex instead. Then we expect an optimal stopping strategy to be one which maximizes local time accumulated at the origin. As expected, we find that the density of $W_{10}$ conditional upon $\tau=10$ is largely concentrated on points away from $x=0$, at which the payoff process is unlikely to spend significant time as a submartingale if we were to choose not to stop.

It is interesting to note the lack of sharp cut-off between the two density estimates. One might expect the optimal strategy to be of a form where there exists a ``stopping region'' and a ``continuation region.'' On the contrary, the smooth overlap of the two density estimates is persistent even as we vary the resolution of the finite-difference solver, which suggests that the true optimal stopping strategy is not of the form $\{\tau=10\}\subset\sigma(W_{10})$. That is, the numerics suggest that optimal stopping strategies may be path-dependent even in simple examples.


\appendix

\section{Proof of Lemma \ref{Lem:WeakDPPLemma}}\label{Appendix:WeakDPPLemma}

This first argument is in the spirit of proofs of the weak dynamic programming principle which avoid measurable selection, as in \cite{BouchardTouzi2011,BouchardNutz2012,BayraktarYao2013}. In these arguments, the authors typically use a covering argument to find a countable selection of $\epsilon$-optimal controls on small balls of the state-space. The main difficulty here is that, while a control may be admissible for the state-constrained problem at one point in state-space, there is no reason to expect it to satisfy the state constraints starting from nearby states.

The new idea in our approach is to cover $\Delta_{k+1}$ with a finite mesh. We show that we can replace the process $Y$ by a modified process $Y^\epsilon$, which lies on the mesh points almost-surely at the terminal time. We construct the new process in a measurable way using the Martingale Representation Theorem on a carefully constructed random variable. Then we show that, using the continuity of $v_{k+1}$, that the objective function along $Y$ is close to that along $Y^\epsilon$ for a fine enough grid.

Once we know we can consider a perturbed process $Y^\epsilon$ which lies on a finite number of points in $\Delta_{k+1}$ at the terminal time almost-surely, we can construct $\epsilon$-optimal stopping times using a standard covering argument in $\mathbb{R}$.

\begin{proof}
Fix $(x,y)\in\mathbb{R}\times\Delta_k$. For convenience of notation, define 
\begin{equation}\nonumber\begin{array}{rccl}
A & := & \sup\limits_{\alpha\in\mathcal{A}_{t_{k-1}}} & \mathbb{E}\left[Y^{(k),t_{k-1},y,\alpha}_{t_k} f(X^{t_{k-1},x}_{t_k})+(1-Y^{(k),t_{k-1},y,\alpha}_{t_k} )v_{k+1}\left(X^{t_{k-1},x}_{t_k},Y^{t_{k-1},y,\alpha}_{t_k}\right)\right] \\
& & \text{s.t.} & Y^{t_{k-1},y,\alpha}_u \in \Delta_k\text{ for all }u\geq t_{k-1}\\
& & & Y^{(k),t_{k-1},y,\alpha}_{t_k} \in \{0,1\}\text{, almost-surely.}
\end{array}\end{equation}
In the remainder of this proof, we do not write $t_{k-1}$ in the superscripts of $X$ and $Y$ because it is always fixed.

\vspace{5mm}\textit{Step 1:} Fix an arbitrary $\epsilon > 0$. Choose $R>0$ large enough that
\[\mathbb{P}\left[\sup\limits_{t_{k-1}\leq u\leq t_k} \left| W_u - W_{t_{k-1}}\right| \geq R\right] \leq \epsilon^2.\]
Because $v_{k+1}$ is continuous on the compact set $[x-R,x+R]\times\Delta_{k+1}$, we can find $\delta > 0$ small enough that
\[\left|v_{k+1}(x',y')-v_{k+1}(x',y'')\right|\leq\epsilon\]
for all $x'\in[x-R,x+R]$ and $y',y''\in\Delta_{k+1}$ such that
\[\|y'-y''\|_{\ell^\infty} \leq \delta.\]
Similarly, because $f$ is Lipschitz and $v_{k+1}$ is Lipschitz in $x$ uniformly in $y$, we can find $\delta > 0$, possibly smaller than before, such that we also have
\[\left|f(x')-f(x'')\right|+\left|v_{k+1}(x',y')-v_{k+1}(x'',y')\right|\leq\epsilon\]
for all $x',x''\in\mathbb{R}$ and $y'\in\Delta_{k+1}$ such that
\[\left|x'-x''\right| \leq \delta.\]

\vspace{5mm}\textit{Step 2:} We now construct a finite mesh on $\Delta_{k+1}$. Let $\mathcal{P} := \{y_j\}_{j=1}^N$ be a finite subset of $\Delta_{k+1}$ with the property that
\begin{itemize}
\item The convex hull of $\mathcal{P}$ is $\Delta_{k+1}$, and
\item Any point $y\in\Delta_{k+1}$ can be written as a convex combination of finitely-many points in $\mathcal{P}$, each contained in a $\delta$-neighborhood of $y$.
\end{itemize}
This is possible by compactness and convexity of $\Delta_{k+1}$. In particular, we can define a continuous function $T : \Delta_{k+1} \to [0,1]^N$ with the properties that
\begin{itemize}
\item $T_j(y) = 0$ for all $y\in\Delta_{k+1}$ such that $|y-y_j|>\delta$
\item $\sum_{j=1}^N T_j(y) = 1$ for all $y\in\Delta_{k+1}$, and
\item $\sum_{j=1}^N y_j T_j(y) = y$ for all $y\in\Delta_{k+1}$.
\end{itemize}
This corresponds to a continuous map from a point $y\in\Delta_{k+1}$ to a probability weighting of points in $\mathcal{P}$ such that $y$ is a convex combination of nearby points in $\mathcal{P}$. Such a map can be obtained by an $\ell^2$-minimization problem, for instance.

\vspace{5mm}\textit{Step 3:} Let $\{A_i\}_{i\geq 1}$ be a countable and disjoint covering of $\mathbb{R}$ with an associated set of points $\{x_i\}$ such that the $\delta$-ball centered at $x_i$ contains the set $A_i$.

For each $i\geq 1$ and $j\in\{1,\ldots,N\}$, let $\tau_{i,j} \in \mathcal{T}_{t_k}$ be a stopping time satisfying
\[\tau_{i,j} \sim \sum_{\ell=1}^r y_j^{(\ell)} \delta_{t_\ell}\]
such that
\[\mathbb{E}\left[f(X^{t_k,x_i}_{\tau_{i,j}})\right] \geq v_{k+1}(x_i,y_j) - \epsilon.\] 
Note that above uses $y_j^{(\ell)}$ to denote the $\ell$th entry of the vector $y_j$.

By choice of $\delta>0$ in the first step and the definition of the sets $A_i$, we have
\begin{eqnarray}
v_{k+1}(x_i,y_j) & \geq & v_{k+1}(x,y_j) - \epsilon \nonumber\\
\mathbb{E}\left[f(X^{t_k,x}_{\tau_{i,j}})\right] & \geq & \mathbb{E}\left[f(X^{t_k,x_i}_{\tau_{i,j}})\right] - \epsilon\nonumber
\end{eqnarray}
for all $x \in A_i$.

Putting these inequalities together, we conclude that
\[\mathbb{E}\left[f(X^{t_k,x}_{\tau_{i,j}})\right] \geq v_{k+1}(x,y_j) - 3\epsilon\]
for all $i\geq 1$, $j\in\{1,\ldots,N\}$, and $x\in A_i$.

%

\vspace{5mm}\textit{Step 4:} Let $\alpha\in\mathcal{A}_{t_{k-1}}$ be an arbitrary control for which $Y^{y,\alpha}_u \in \Delta_k$ for $u\geq t_{k-1}$ and $Y^{(k),y,\alpha}_{t_k} \in \{0,1\}$ almost-surely. For any $0<h<t_k-t_{k-1}$, define two random variables, $M_1$ and $M_2$, as
\begin{eqnarray}\label{eq:defnM12}
M_1 &  := & h^{-1/2}\left(W_{t_k}-W_{t_k-h}\right)\\
M_2 & := & h^{-1/2}\max_{t_k-h\leq s\leq t_k}|W_s-W_{t_k-h}-\delta^{-1}\left(s-t_k+h\right)\left(W_{t_k}-W_{t_k-h}\right)|.\nonumber
\end{eqnarray}
Then $M_1$ and $M_2$ are $\mathcal{F}_{t_k}$-measurable and independent of each other. $M_1$ is equal in distribution to a standard normal distribution, the cumulative distribution function of which we denote by $\Phi$. Similarly, $M_2$ is equal in distribution to the absolute maximum of a standard Brownian bridge on $[0,1]$, the cumulative distribution function of which we denote by $\Phi_{BB}$. Furthermore, if we define
\[\mathcal{G} := \sigma\left(\mathcal{F}_{t_k-h}\cup\sigma(W_{t_k})\right),\]
then $M_1$ is $\mathcal{G}$-measurable, while $M_2$ is independent of $\mathcal{G}$.

Define a random vector $\overline{Y}_{t_k}$ as
\begin{equation}\nonumber
\overline{Y}^{(k)}_{t_k} := 1_{\left\{M_2 \leq \Phi^{-1}_{BB}\left(Y^{(k),y,\alpha}_{t_k-h}\right)\right\}}
\end{equation}
and
\begin{equation}\nonumber\begin{array}{l}
\overline{Y}^{(k+1):r}_{t_k} := 1_{\{M_2 > \Phi^{-1}_{BB}(Y^{(k),y,\alpha}_{t_k-h})\}} \times\\
\hspace{3cm}\sum\limits_{j=1}^N y_j 1_{\{\Phi^{-1}(\sum_{i=1}^{j-1} T_i(P_k(Y^{y,\alpha}_{t_k-h})))<M_1\leq\Phi^{-1}(\sum_{i=1}^j T_i(P_k(Y^{y,\alpha}_{t_k-h})))\}},
\end{array}
\end{equation}
where we follow the conventions that $\Phi^{-1}(0)=-\infty$, $\Phi^{-1}(1)=+\infty$, and that sums over an empty set are zero. We denote the $(k+1)$th through $r$th entry in the random vector by $\overline{Y}^{(k+1):r}_{t_k}$. Then $\overline{Y}_{t_k} \in \Delta_k$ is $\mathcal{F}_{t_k}$-measurable and is constructed to have the key property that
\[\mathbb{E}\left[\overline{Y}_{t_k}\mid\mathcal{F}_{t_k-h}\right] = Y^{y,\alpha}_{t_k-h},\]
almost-surely.

By the Martingale Representation Theorem, there exists $\alpha_\epsilon\in\mathcal{A}$ for which  $Y^{y,\alpha_\epsilon}_{t_k} = \overline{Y}_{t_k}$ almost-surely. It is clear by the construction that $\overline{Y}_{t_k}$ is independent of $\mathcal{F}_{t_{k-1}}$, so we can take $\alpha_\epsilon \in \mathcal{A}_{t_{k-1}}$. Then, by construction, $Y^{y,\alpha_\epsilon}_u \in \Delta_k$ for all $u\geq t_{k-1}$, $Y^{(k),y,\alpha_\epsilon}_{t_k} \in \{0,1\}$, and $Y^{y,\alpha_\epsilon}_{t_k} \in \mathcal{P}$ when $Y^{(k),y,\alpha_\epsilon}_{t_k}=0$, almost-surely.

We now perform a key computation. First note that
\begin{equation}\nonumber\begin{array}{l}
\mathbb{E}\left[Y^{(k),y,\alpha_\epsilon}_{t_k} f(X^{x}_{t_k}) + (1-Y^{(k),y,\alpha_\epsilon}_{t_k})v_{k+1}\left(X^{x}_{t_k},Y^{y,\alpha_\epsilon}_{t_k}\right)\right] \\
\hspace{1cm} = \mathbb{E}\left[1_{\{\overline{Y}^{(k)}_{t_k}=1\}} f(X^{x}_{t_k})\right] + \mathbb{E}\left[1_{\{\overline{Y}^{(k)}_{t_k}=0\}}v_{k+1}\left(X^{x}_{t_k},\overline{Y}_{t_k}\right)\right].
\end{array}\end{equation}
For the first term on the right-hand-side, we simply compute
\begin{equation}\nonumber\begin{array}{l}
\mathbb{E}\left[1_{\{\overline{Y}^{(k)}_{t_k}=1\}} f(X^{x}_{t_k})\right] = \mathbb{E}\left[1_{\{M_2 \leq \Phi^{-1}_{BB}(Y^{(k),y,\alpha}_{t_k-h})\}} f(X^{x}_{t_k})\right] \\
\hspace{2cm} = \mathbb{E}\left[\mathbb{E}\left[1_{\{M_2 \leq \Phi^{-1}_{BB}(Y^{(k),y,\alpha}_{t_k-h})\}}\mid\mathcal{G}\right] f(X^{x}_{t_k})\right] \\
\hspace{2cm} = \mathbb{E}\left[Y^{(k),y,\alpha}_{t_k-h}f(X^{x}_{t_k})\right].
\end{array}\end{equation}
We deal with the second term in a similar way, but the computation is more involved. Note that by construction we have
\[\|\overline{Y}_{t_k}-P_k(Y^{t,y,\alpha}_{t_k-h})\|_{\ell^\infty}\leq\delta\]
almost-surely in the set $\{\overline{Y}^{(k)}_\theta = 0\}$. Recall we also took $\delta$ small enough such that
\[|v_{k+1}(x',y')-v_{k+1}(x',y'')|\leq\epsilon\]
for all $x'\in[x-R,x+R]$ and $y',y''\in\Delta_{k+1}$ such that $\|y'-y''\|_{\ell^\infty}\leq\delta$. But then we can compute
\begin{equation}\nonumber\begin{array}{l}
\mathbb{E}\left[1_{\{\overline{Y}^{(k)}_{t_k}=0\}}v_{k+1}(X^{x}_{t_k},\overline{Y}_{t_k})\right] \\
\hspace{1cm} = \mathbb{E}\left[1_{\{\overline{Y}^{(k)}_{t_k}=0\}}1_{\{|W_{t_k}|\leq R\}}v_{k+1}(X^{x}_{t_k},\overline{Y}_{t_k})\right] + \mathbb{E}\left[1_{\{\overline{Y}^{(k)}_{t_k}=0\}}1_{\{|W_{t_k}|\geq R\}}v_{k+1}(X^{x}_{t_k},\overline{Y}_{t_k})\right] \\
\hspace{1cm} \geq \mathbb{E}\left[1_{\{\overline{Y}^{(k)}_{t_k}=0\}}1_{\{|W_{t_k}|\leq R\}}v_{k+1}(X^{x}_{t_k},P_k(Y^{y,\alpha}_{t_k-h}))\right] \\
\hspace{2cm} + \mathbb{E}\left[1_{\{\overline{Y}^{(k)}_{t_k}=0\}}1_{\{|W_{t_k}|\geq R\}}v_{k+1}(X^{x}_{t_k},\overline{Y}_{t_k})\right] - \epsilon \\
\hspace{1cm} \geq \mathbb{E}\left[1_{\{\overline{Y}^{(k)}_{t_k}=0\}}v_{k+1}(X^{x}_{t_k},P_k(Y^{y,\alpha}_{t_k-h}))\right] \\
\hspace{2cm}- \mathbb{E}\left[1_{\{|W_{t_k}|\geq R\}}\left(\left|v_{k+1}(X^{x}_{t_k},\overline{Y}_{t_k})\right|+\left|v_{k+1}(X^{x}_{t_k},Y^{y,\alpha}_{t_k-h})\right|\right)\right] - \epsilon\\
\hspace{1cm} \geq \mathbb{E}\left[1_{\{\overline{Y}^{(k)}_{t_k}=0\}}v_{k+1}(X^{x}_{t_k},P_k(Y^{y,\alpha}_{t_k-h}))\right] - 2\sqrt{\mathbb{P}\left[|W_{t_k}|\geq R\right]}\,C(1+|x|) - \epsilon \\
\hspace{1cm}\geq \mathbb{E}\left[1_{\{\overline{Y}^{(k)}_{t_k}=0\}}v_{k+1}(X^{x}_{t_k},P_k(Y^{y,\alpha}_{t_k-h}))\right] - 2(1+C)(1+|x|)\epsilon,
\end{array}\end{equation}
where $C>0$ comes from the growth bound from Proposition~\ref{Proposition:LipschitzEstimate}. With this in hand, we now complete the analysis of the second term
\begin{eqnarray}
\mathbb{E}\left[1_{\{\overline{Y}^{(k)}_{t_k}=0\}}v_{k+1}(X^{x}_{t_k},P_k(Y^{y,\alpha}_{t_k-h}))\right] & = & \mathbb{E}\left[1_{\{M_2 > \Phi^{-1}_{BB}\left(Y^{(k),y,\alpha}_{t_k-h}\right)\}}v_{k+1}(X^{x}_{t_k},P_k(Y^{y,\alpha}_{t_k-h}))\right] \nonumber\\
& = & \mathbb{E}\left[\mathbb{E}\left[1_{\{M_2 > \Phi^{-1}_{BB}\left(Y^{(k),y,\alpha}_{t_k-h}\right)\}}\mid\mathcal{G}\right]\,v_{k+1}(X^{x}_{t_k},P_k(Y^{y,\alpha}_{t_k-h}))\right] \nonumber\\
& = & \mathbb{E}\left[(1-Y^{(k),y,\alpha}_{t_k-h})v_{k+1}(X^{x}_{t_k},P_k(Y^{y,\alpha}_{t_k-h}))\right]. \nonumber
\end{eqnarray}
Using the continuity of $f$, $v_{k+1}$, and $P_k$, along with the Dominated Convergence Theorem, we note
\begin{equation}\nonumber\begin{array}{l}
\lim\limits_{h\to 0}\mathbb{E}\left[Y^{(k),y,\alpha}_{t_k-h}f(X^{x}_{t_k}) + (1-Y^{(k),y,\alpha}_{t_k-h})v_{k+1}\left(X^{x}_{t_k},P_k(Y^{y,\alpha}_{t_k-h})\right)\right] \\
\hspace{2cm} = \mathbb{E}\left[Y^{(k),y,\alpha}_{t_k} f(X^{x}_{t_k}) + (1-Y^{(k),y,\alpha}_{t_k})v_{k+1}\left(X^{x}_{t_k},P_k(Y^{y,\alpha}_{t_k})\right)\right] \\
\hspace{2cm} = \mathbb{E}\left[Y^{(k),y,\alpha}_{t_k} f(X^{x}_{t_k}) + (1-Y^{(k),y,\alpha}_{t_k})v_{k+1}\left(X^{x}_{t_k},Y^{y,\alpha}_{t_k}\right)\right].
\end{array}\end{equation}
Then putting these results together, we see that for $h>0$ small enough
\begin{equation}\nonumber\begin{array}{l}
\mathbb{E}\left[Y^{(k),y,\alpha_\epsilon}_{t_k} f(X^{x}_{t_k}) + (1-Y^{(k),y,\alpha_\epsilon}_{t_k})v_{k+1}\left(X^{x}_{t_k},Y^{y,\alpha_\epsilon}_{t_k}\right)\right] \\
\hspace{2cm} \geq \mathbb{E}\left[Y^{(k),y,\alpha}_{t_k} f(X^{x}_{t_k}) + (1-Y^{(k),y,\alpha}_{t_k})v_{k+1}\left(X^{x}_{t_k},Y^{y,\alpha}_{t_k}\right)\right] - \epsilon-2(1+C)(1+|x|)\epsilon.
\end{array}\end{equation}

\vspace{5mm}\textit{Step 5:} Lastly, we intend to construct an $\epsilon$-optimal stopping time using the covering from the second step. Define a stopping time $\tau_\epsilon$ as
\[\tau_\epsilon := t_k + 1_{\{Y^{(k),t,y,\alpha_\epsilon}_{t_k}=0\}}\sum\limits_{i=1}^\infty\sum\limits_{j=1}^N \tau_{i,j}1_{\{X^{x}_{t_k} \in A_i\}}1_{\{Y^{y,\alpha_\epsilon}_{t_k} = y_j\}}.\]
By construction, we have $\tau_\epsilon \sim \sum_{\ell=1}^r y_\ell \delta_{t_\ell}$. We proceed to make a careful computation. First, note that
\begin{equation}\nonumber
\mathbb{E}\left[f(X^{x}_{\tau_\epsilon})\right] = \mathbb{E}\left[1_{\{\tau_\epsilon=t_k\}}f(X^{x}_{t_k})\right] + \mathbb{E}\left[1_{\{\tau_\epsilon>t_k\}}f(X^{x}_{\tau_\epsilon})\right].
\end{equation}
We focus on the second term. In particular, we have
\begin{equation}\nonumber\begin{array}{l}
\mathbb{E}\left[1_{\{\tau_\epsilon>t_k\}}f(X^{x}_{\tau_\epsilon})\right] = \sum\limits_{i=1}^\infty\sum\limits_{j=1}^N \mathbb{E}\left[1_{\{Y^{(k),y,\alpha_\epsilon}_{t_k}=0\}}1_{\{X^{x}_{t_k} \in A_i\}}1_{\{Y^{y,\alpha_\epsilon}_{t_k}=y_j\}} f(X^{x}_{\tau_{i,j}})\right] \\
\hspace{1cm}= \sum\limits_{i=1}^\infty\sum\limits_{j=1}^N \mathbb{E}\left[1_{\{Y^{(k),y,\alpha_\epsilon}_{t_k}=0\}}1_{\{X^{x}_{t_k} \in A_i\}}1_{\{Y^{y,\alpha_\epsilon}_{t_k}=y_j\}} \mathbb{E}\left[f(X^{x}_{\tau_{i,j}})\mid\mathcal{F}_{t_k}\right]\right]  \\
\hspace{1cm}\geq \sum\limits_{i=1}^\infty\sum\limits_{j=1}^N \mathbb{E}\left[1_{\{Y^{(k),y,\alpha_\epsilon}_{t_k}=0\}}1_{\{X^{x}_{t_k} \in A_i\}}1_{\{Y^{y,\alpha_\epsilon}_{t_k}=y_j\}} v_{k+1}(X^{x}_{t_k},Y^{y,\alpha_\epsilon}_{t_k})\right] - 3\epsilon \\
\hspace{1cm} = \mathbb{E}\left[(1-Y^{(k),y,\alpha_\epsilon}_{t_k})v_{k+1}(X^{x}_{t_k},Y^{y,\alpha_\epsilon}_{t_k})\right] - 3\epsilon,
\end{array}\end{equation}
where the inequality follows from the construction in the third step and the independence of the stopping times $\tau_{ij}$ with respect to $\mathcal{F}_{t_k}$.

Then we conclude
\[\mathbb{E}\left[f(X^{x}_{\tau_\epsilon})\right] \geq \mathbb{E}\left[Y^{(k),y,\alpha_\epsilon}_{t_k} f(X^{x}_{t_k}) + (1-Y^{(k),y,\alpha_\epsilon}_{t_k})v_{k+1}(X^{x}_{t_k},Y^{y,\alpha_\epsilon}_{t_k})\right] - 3\epsilon.\]
Combining this with the main inequality from the previous step, we obtain
\begin{eqnarray}
v_k(x,y) & \geq & \mathbb{E}\left[f(X^{x}_{\tau_\epsilon})\right]\nonumber\\
& \geq & \mathbb{E}\left[Y^{(k),y,\alpha_\epsilon}_{t_k} f(X^{x}_{t_k}) + (1-Y^{(k),y,\alpha_\epsilon}_{t_k})v_{k+1}(X^{x}_{t_k},Y^{y,\alpha_\epsilon}_{t_k})\right] - 3\epsilon \nonumber\\
& \geq & \mathbb{E}\left[Y^{(k),y,\alpha}_{t_k} f(X^{x}_{t_k}) + (1-Y^{(k),y,\alpha}_{t_k})v_{k+1}\left(X^{x}_{t_k},Y^{y,\alpha}_{t_k}\right)\right] - 4\epsilon - 2(1+C)(1+|x|)\epsilon.\nonumber
\end{eqnarray}
Because $\epsilon$ and $\alpha$ were arbitrary, then we conclude $A\leq v_k(x,y)$.

\vspace{5mm}\textit{Step 6:} Let $\tau\in\mathcal{T}_{t_k}$ be an arbitrary stopping time such that $\tau \sim \sum_{\ell=1}^r y_\ell \delta_{t_\ell}$. Define a martingale as
\begin{equation}\nonumber
Y^{(i)}_t := \mathbb{E}\left[1_{\{\tau=t_i\}}\mid\mathcal{F}_t\right]
\end{equation}
for all $t\geq 0$ and each $i\in\{1,\ldots,r\}$. We can easily check that $Y^{(i)}_0 = y_i$ for each $i\in\{1,\ldots,r\}$ and
\begin{equation}\nonumber
Y^{(1)}_t + \cdots + Y^{(r)}_t = \mathbb{E}\left[1_{\{\tau=t_1\}}+\cdots+1_{\{\tau=t_r\}}\mid\mathcal{F}_t\right] = 1.
\end{equation}
Then if we consider $Y$ as an $\mathbb{R}^r$-valued martingale with $Y^{(i)}_t \equiv 0$ for all $i\in\{1,\ldots,k-1\}$, then we see $Y_t \in \Delta_k$ for each $t\geq 0$. Finally, we have
\begin{equation}\nonumber
Y^{(k)}_{t_k} = \mathbb{E}\left[1_{\{\tau=t_k\}}\mid\mathcal{F}_{t_k}\right] = 1_{\{\tau=t_k\}} \in \{0,1\}.
\end{equation}

Then by the Martingale Representation Theorem, there exists $\alpha\in\mathcal{A}_{t_k}$ for which $Y^{y,\alpha}_t = Y_t$ for all $t\geq 0$, almost-surely. We can compute
\begin{eqnarray}
\mathbb{E}\left[f(X^x_\tau)\right] & = & \mathbb{E}\left[1_{\{\tau=t_k\}}f(X^x_{t_k})+1_{\{\tau>t_k\}}f(X^x_\tau)\right]\nonumber\\
& = & \mathbb{E}\left[Y^{(k),y,\alpha}_{t_k} f(X^x_{t_k})+(1-Y^{(k),y,\alpha}_{t_k} )\mathbb{E}\left[f(X^x_\tau)\mid\mathcal{F}_{t_k}\right]\right].\nonumber
\end{eqnarray}
On the set $\{\tau>t_k\}$, we have
\begin{equation}\nonumber
\mathbb{P}\left[\tau=t_i\mid\mathcal{F}_{t_k}\right] = \mathbb{E}\left[1_{\{\tau=t_i\}}\mid\mathcal{F}_{t_k}\right] = Y^{(i)}_{t_k}
\end{equation}
for each $i\in\{k+1,\ldots,r\}$. For almost every $\omega\in\{\tau>t_k\}$, we have
\begin{equation}\nonumber
\mathbb{E}\left[f(X^x_\tau)\mid\mathcal{F}_{t_k}\right] \leq v_{k+1}\left(X^x_{t_k},Y^{y,\alpha}_{t_k}\right)
\end{equation}
by the Strong Markov Property and stationarity properties of Brownian motion. Then we conclude
\[\mathbb{E}\left[f(X^x_\tau)\right] \leq \mathbb{E}\left[Y^{(k),y,\alpha}_{t_k} f(X^x_{t_k})+(1-Y^{(k),y,\alpha}_{t_k} )v_{k+1}(X^x_{t_k},Y^{y,\alpha}_{t_k})\right] \leq A.\]
Because $\tau$ was an arbitrary stopping time, this implies
\begin{equation}\nonumber
v_k(x,y) \leq A
\end{equation}

\end{proof}

\section{Proof of Lemma~\ref{Lem:TerminalRelaxation}}\label{Appendix:ProofTerminalRelaxation}

The main idea of this argument is that we can take a controlled process $Y$, which does not satisfy $Y^{(k)}_{t_k}\in\{0,1\}$, and modify it on an interval $[t_k-h,t_k]$ to a perturbed process $Y^\epsilon$ with the properties that $Y_{t_k-h}=Y^\epsilon_{t_k-h}$ and $Y^{\epsilon,(k)}_{t_k}\in\{0,1\}$. In particular, we may do this in a way that does not appreciably change the expected payoff.

One key idea, which we draw the reader's attention towards, is the use of the Brownian bridge over $[t_k-h,t_k]$ in the construction. This construction is in the spirit of Corollary \ref{Cor:ExistenceConstruction2}. While one might initially attempt a construction similar to Corollary \ref{Cor:ExistenceConstruction}, using a Brownian bridge instead of Brownian increments allows us to condition on $W_{t_k}$ at a key point in the argument.

\begin{proof}
Fix $(x,y)\in\mathbb{R}\times\Delta_k$. For convenience of notation, define
\begin{equation}\nonumber\begin{array}{rccl}
A & := & \sup\limits_{\alpha\in\mathcal{A}_{t_{k-1}}} & \mathbb{E}\left[Y^{(k),t_{k-1},y,\alpha}_{t_k} f(X^{t_{k-1},x}_{t_k})+(1-Y^{(k),t_{k-1},y,\alpha}_{t_k} )v_{k+1}\left(X^{t_{k-1},x}_{t_k},Y^{t_{k-1},y,\alpha}_{t_k}\right)\right] \\
& & \text{s.t.} & Y^{t_{k-1},y,\alpha}_u \in \Delta_k\text{ for }u\geq t_{k-1}\\
& & & Y^{(k),t_{k-1},y,\alpha}_{t_k} \in \{0,1\}\text{ almost-surely},
\end{array}\end{equation}
and
\begin{equation}\nonumber\begin{array}{rccl}
B & := & \sup\limits_{\alpha\in\mathcal{A}_{t_{k-1}}} & \mathbb{E}\left[Y^{(k),t_{k-1},y,\alpha}_{t_k} f(X^{t_{k-1},x}_{t_k})+(1-Y^{(k),t_{k-1},y,\alpha}_{t_k} )v_{k+1}\left(X^{t_{k-1},x}_{t_k},P_k(Y^{t_{k-1},y,\alpha}_{t_k})\right)\right] \\
& & \text{s.t.} & Y^{t_{k-1},y,\alpha}_u \in \Delta_k\text{ for }u\geq t_{k-1}.
\end{array}\end{equation}
By Lemma \ref{Lem:WeakDPPLemma}, we have $v_k(x,y) = A$. In the remainder of the proof we withold the superscript $t_{k-1}$ on $X$ and $Y$ for the sake of brevity. 

\vspace{5mm}\textit{Step 1:} Let $\alpha\in\mathcal{A}_{t_{k-1}}$ be an arbitrary control for which $Y^{y,\alpha}_u \in \Delta_k$ for $u\geq t_{k-1}$ and $Y^{(k),y,\alpha}_{t_k}\in\{0,1\}$ almost-surely. Note that $Y^{y,\alpha}_{t_k} = P_k(Y^{y,\alpha}_{t_k})$ on the set $\{Y^{(k),y,\alpha}_{t_k}=0\}$, almost-surely. Then
\begin{equation}\nonumber\begin{array}{l}
\mathbb{E}\left[Y^{(k),y,\alpha}_{t_k} f(X^x_{t_k})+(1-Y^{(k),y,\alpha}_{t_k} )v_{k+1}\left(X^x_{t_k},Y^{y,\alpha}_{t_k}\right)\right] \\
\hspace{2cm}= \mathbb{E}\left[Y^{(k),y,\alpha}_{t_k} f(X^x_{t_k})+(1-Y^{(k),y,\alpha}_{t_k} )v_{k+1}\left(X^x_{t_k},P_k(Y^{y,\alpha}_{t_k})\right)\right]\\
\hspace{2cm} \leq B.
\end{array}\end{equation}
Because $\alpha$ was arbitrary, we conclude $A\leq B$.

\vspace{5mm}\textit{Step 2:} Let $\alpha\in\mathcal{A}_{t_{k-1}}$ be an arbitrary control for which $Y^{y,\alpha}_u\in\Delta_k$ for $u\geq t_{k-1}$, almost-surely. For any $0<h<t_k-t_{k-1}$, define a random variable $M$ as
\[M := h^{-1/2}\max_{t_k-h\leq s\leq t_k}|W_s-W_{t_k-h}-h^{-1}\left(s-t_k+h\right)\left(W_{t_k}-W_{t_k-h}\right)|.\]
Then $M$ is $\mathcal{F}_{t_k}$-measurable and is equal in distribution to the absolute maximum of a standard Brownian bridge on $[0,1]$, the cumulative distribution function of which we denote by $\Phi_{BB}$. If we define $\mathcal{G} := \sigma\left(\mathcal{F}_{t_k-h}\cup\sigma(W_{t_k})\right)$, then $M$ is independent of $\mathcal{G}$.

Define a random vector $\overline{Y}_{t_k}$ as
\begin{equation}\nonumber
\overline{Y}^{(k)}_{t_k} := 1_{\{M \leq \Phi_{BB}^{-1}\left(Y^{(k),y,\alpha}_{t_k-h}\right)\}}
\end{equation}
and
\begin{equation}\nonumber
\overline{Y}^{(k+1):r}_{t_k} := P_k(Y^{y,\alpha}_{t_k-h})1_{\{M > \Phi_{BB}^{-1}\left(Y^{(k),y,\alpha}_{t_k-h}\right)\}},
\end{equation}
where $\overline{Y}^{(k+1):r}_{t_k}$ denotes the $(k+1)$th through $r$th element in the vector. Let $\overline{Y}^{(i)}_{t_k} \equiv 0$ for any $i\in\{1,\ldots,k-1\}$. Then $\overline{Y}_{t_k}$ is $\mathcal{F}_{t_k}$-measurable and has the key property that $\mathbb{E}\left[\overline{Y}_{t_k}\mid\mathcal{F}_{t_k-h}\right] = Y^{y,\alpha}_{t_k-h}$. We also note that $\mathbb{E}\left[1_{\{\overline{Y}^{(k)}_{t_k}=1\}}\mid\mathcal{G}\right] = Y^{(k),y,\alpha}_{t_k-h}$.

By the Martingale Representation Theorem, there exists $\alpha_\epsilon\in\mathcal{A}_{t_{k-1}}$ such that $Y^{y,\alpha_\epsilon}_u \in \Delta_k$ for $u\geq t_{k-1}$, $Y^{(k),y,\alpha_\epsilon}_{t_k} \in \{0,1\}$, and $Y^{y,\alpha_\epsilon}_{t_k} = \overline{Y}_{t_k}$ almost-surely. We can then compute
\begin{equation}\nonumber\begin{array}{l}
\mathbb{E}\left[Y^{(k),y,\alpha_\epsilon}_{t_k} f(X^x_{t_k})+(1-Y^{(k),y,\alpha_\epsilon}_{t_k})v_{k+1}\left(X^x_{t_k},Y^{y,\alpha_\epsilon}_{t_k}\right)\right] \\
\hspace{2cm} = \mathbb{E}\left[1_{\{\overline{Y}^{(k)}_{t_k}=1\}} f(X^x_{t_k})+1_{\{\overline{Y}^{(k)}_{t_k}=0\}}v_{k+1}\left(X^x_{t_k},P_k(Y^{y,\alpha}_{t_k-h})\right)\right]  \\
\hspace{2cm} = \mathbb{E}\left[\mathbb{E}\left[1_{\{\overline{Y}^{(k)}_{t_k}=1\}}\mid\mathcal{G}\right] f(X^x_{t_k})+\mathbb{E}\left[1_{\{\overline{Y}^{(k)}_{t_k}=0\}}\mid\mathcal{G}\right]v_{k+1}\left(X^x_{t_k},P_k(Y^{y,\alpha}_{t_k-h})\right)\right] \\
\hspace{2cm} = \mathbb{E}\left[Y^{(k),y,\alpha}_{t_k-h} f(X^x_{t_k})+(1-Y^{(k),y,\alpha}_{t_k-h})v_{k+1}\left(X^x_{t_k},P_k(Y^{y,\alpha}_{t_k-h})\right)\right] .
\end{array}\end{equation}

But by the continuity and growth bounds of $f$ and $v_{k+1}$, we can apply the Dominated Convergence Theorem to see
\begin{equation}\nonumber\begin{array}{l}
\lim\limits_{\delta\to 0^+}\mathbb{E}\left[Y^{(k),y,\alpha}_{t_k-h} f(X^x_{t_k})+(1-Y^{(k),y,\alpha}_{t_k-h})v_{k+1}\left(X^x_{t_k},P_k(Y^{y,\alpha}_{t_k-h})\right)\right]\\
\hspace{2cm} = \mathbb{E}\left[Y^{(k),y,\alpha}_{t_k} f(X^x_{t_k})+(1-Y^{(k),y,\alpha}_{t_k})v_{k+1}\left(X^x_{t_k},P_k(Y^{y,\alpha}_{t_k})\right)\right].
\end{array}\end{equation}
So then for any $\epsilon>0$, we may take $h>0$ small enough that
\begin{equation}\nonumber\begin{array}{l}
\mathbb{E}\left[Y^{(k),y,\alpha}_{t_k} f(X^x_{t_k})+(1-Y^{(k),y,\alpha}_{t_k})v_{k+1}\left(X^x_{t_k},P_k(Y^{y,\alpha}_{t_k})\right)\right] \\
\hspace{2cm} \leq \mathbb{E}\left[Y^{(k),y,\alpha}_{t_k-h} f(X^x_{t_k})+(1-Y^{(k),y,\alpha}_{t_k-h})v_{k+1}\left(X^x_{t_k},P_k(Y^{y,\alpha}_{t_k-h})\right)\right] + \epsilon \\
\hspace{2cm} = \mathbb{E}\left[Y^{(k),y,\alpha_\epsilon}_{t_k} f(X^x_{t_k})+(1-Y^{(k),y,\alpha_\epsilon}_{t_k})v_{k+1}\left(X^x_{t_k},Y^{y,\alpha_\epsilon}_{t_k}\right)\right] +\epsilon \\
\hspace{2cm} \leq A + \epsilon.
\end{array}\end{equation}
Because $\epsilon$ and $\alpha$ were arbitrary, we conclude $B\leq A$.

\end{proof}

\section{Proof of Proposition~\ref{Proposition:TimeDependentHolderEstimates}}\label{Appendix:TimeDependentHolderEstimates}

\begin{proof}

We proceed in several steps, each relating the value function between nearby points. In the first three points, we consider a shift backwards in the time variable, a shift forward in time variable away from the terminal time, and lastly a jump onto the terminal time. In the fourth step, we discuss arbitrary perturbations in $x$. In the fifth step, we discuss a perturbation inside the interior of some face of $\Delta_k$, including a possible jump off the face. In the sixth step, we consider a jump from an interior point onto a face of $\Delta_k$. Lastly, in the final step, we discuss how to put these together into one coherent bound.

\vspace{5mm}\textit{Step 1:} Fix $(t,x,y)\in[t_{k-1},t_k]\times\mathbb{R}\times\Delta_k$ and $t'\in[t_{k-1},t_k]$ such that $t'\leq t$. Let $\alpha\in\mathcal{A}$ be an arbitrary control for which $Y^{t,y,\alpha}_u \in \Delta_k$ for all $u\geq t$, almost-surely. Define a new control $\alpha' \in \mathcal{A}$ as
\[\alpha'_u := 1_{\{u\geq t\}}\alpha_u,\]
for all $u\geq t'$. We can see that $Y^{t',y,\alpha'}_u \in \Delta_k$ for all $u\in[t', t_k]$ and $Y^{t',y,\alpha'}_{t_k} = Y^{t,y,\alpha}_{t_k}$, almost-surely. Then
\begin{eqnarray}
w_k(t',x,y) & \geq & \mathbb{E}\left[Y^{(k),t',y,\alpha'}_{t_k} f(X^{t',x}_{t_k}) + (1-Y^{(k),t',y,\alpha'}_{t_k} ) v_{k+1}(X^{t',x}_{t_k},P_k(Y^{t',y,\alpha'}_{t_k}))\right] \nonumber\\
& = & \mathbb{E}\left[Y^{(k),t,y,\alpha}_{t_k} f(X^{t',x}_{t_k}) + (1-Y^{(k),t,y,\alpha}_{t_k} ) v_{k+1}(X^{t',x}_{t_k},P_k(Y^{t,y,\alpha}_{t_k}))\right] \nonumber\\
& \geq & \mathbb{E}\left[Y^{(k),t,y,\alpha}_{t_k} f(X^{t,x}_{t_k}) + (1-Y^{(k),t,y,\alpha}_{t_k} ) v_{k+1}(X^{t,x}_{t_k},P_k(Y^{t,y,\alpha}_{t_k}))\right] \nonumber\\
& & \hspace{1cm}- 2C\,\mathbb{E}\left|W_{t}-W_{t'}\right|,\nonumber
\end{eqnarray}
where $C>0$ is at least as large as the Lipschitz constants in $x$ for $f$ and $v_{k+1}$. But recall that for Brownian motion we can find $C>0$ such that
\[\mathbb{E}\left|W_t-W_{t'}\right| \leq C\left|t-t'\right|^{1/2}.\]
Using this and the fact that $\alpha$ was arbitrary, we then conclude
\[w_k(t,x,y) - w_k(t',x,y) \leq 2C^2\left|t-t'\right|^{1/2}.\]

\vspace{5mm}\textit{Step 2:} Fix $(t,x,y)\in[t_{k-1},t_k)\times\mathbb{R}\times\Delta_k$ and $t'\in[t_{k-1},t_k)$ such that $t\leq t'$. Define
\[\eta := \sqrt{\frac{t_k-t}{t_k - t'}} \geq 1.\]
Let $\alpha\in\mathcal{A}$ be an arbitrary control for which $Y^{t,y,\alpha'}_u \in \Delta_k$ for all $u\geq t$, almost-surely. Define new control $\alpha'\in\mathcal{A}$ as
\[\alpha'_u := \eta\,\alpha_{\tau_u},\]
where
\[\tau_u := \eta^2(u-t') + t\]
for all $u\in[t',t_k]$. Note that $\alpha'_u \in \mathcal{F}_{\tau_u}$ by definition. Because $\tau_u \leq u$, we then have $\alpha'_u \in \mathcal{F}_u$ so it is an adapted control. We can also check by the time-change properties of the It\^{o} Integral that
\[\left(W_{t_k}-W_{t'},Y^{t',y,\alpha'}_{t_k}\right) \stackrel{(d)}{=} \left(\eta^{-1}(W_{t_k}-W_{t}),Y^{t,y,\alpha}_{t_k}\right).\]
Then $Y^{t',y,\alpha'}_u \in \Delta_k$ for all $u\in[t',t_k]$, almost-surely, by the convexity of $\Delta_k$ and the martingale property of $Y$. Then $\alpha'$ is an admissible control.

We can compute
\begin{eqnarray}
w_k(t',x,y) & \geq & \mathbb{E}\left[Y^{(k),t',y,\alpha'}_{t_k} f(X^{t',x}_{t_k}) + (1-Y^{(k),t',y,\alpha'}_{t_k} ) v_{k+1}(X^{t',x}_{t_k},P_k(Y^{t',y,\alpha'}_{t_k}))\right] \nonumber\\
& \geq & \mathbb{E}\left[Y^{(k),t,y,\alpha}_{t_k} f(X^{t,x}_{t_k}) + (1-Y^{(k),t,y,\alpha'}_{t_k} ) v_{k+1}(X^{t,x}_{t_k},P_k(Y^{t,y,\alpha'}_{t_k}))\right] \nonumber\\
& & \hspace{1cm} - 2C(1-\eta^{-1})\mathbb{E}\left|W_{t_k}-W_{t}\right| \nonumber.
\end{eqnarray}
Now we proceed to bound the final term in this inequality. First, note that by the convexity of $x\mapsto x^{-1/2}$, we can a bound
\[\eta^{-1} = \left(1 + \frac{t'-t}{t_k-t'}\right)^{-1/2} \geq 1 - \frac{t'-t}{2(t_k-t')}.\]
Furthermore, for large enough $C>0$, depending only upon $t_r$, we have $\mathbb{E}\left|W_u\right| \leq C$ for all $u\in[0,t_r]$. Then we can estimate
\[(1-\eta^{-1})\mathbb{E}\left|W_{t_k}-W_{t}\right| \leq 2C\frac{t'-t}{2(t_k-t')}.\]
Putting these together and recalling that $\alpha$ was arbitrary, we conclude
\[w_k(t,x,y) - w_k(t',x,y) \leq 4C^2\frac{\left|t-t'\right|}{t_k-t'}.\]

\vspace{5mm}\textit{Step 3: } Fix $(t,x,y)\in[t_{k-1},t_k]\times\mathbb{R}\times\Delta_k$ and let $\alpha\in\mathcal{A}$ be an arbitrary control for which $Y^{t,y,\alpha}_u \in \Delta_k$ for all $u\geq t$, almost-surely. By the Lipschitz continuity of $f$ and $v_{k+1}$ in $x$ we can compute
\[\begin{array}{l}
\mathbb{E}\left[Y^{(k),t,y,\alpha}_{t_k} f(X^{t,x}_{t_k}) + (1-Y^{(k),t,y,\alpha}_{t_k} ) v_{k+1}(X^{t,x}_{t_k},P_k(Y^{t,y,\alpha}_{t_k}))\right] \\
\hspace{1cm} \leq \mathbb{E}\left[Y^{(k),t,y,\alpha}_{t_k} f(x) + (1-Y^{(k),t,y,\alpha}_{t_k} ) v_{k+1}(x,P_k(Y^{t,y,\alpha}_{t_k}))\right] + C\mathbb{E}\left[\left|W_{t_k}-W_t\right|\right]
\end{array}\]
We can bound the error term by $C\left|t_k-t\right|^{1/2}$. By viewing the term containing $v_{k+1}$ as a perspective map applied to a concave function (See the proof of Proposition~\ref{Proposition:Concavity}) and noting that the controlled process $Y$ is a martingale, we can apply Jensen's Inequality to see
\begin{eqnarray}
\mathbb{E}\left[Y^{(k),t,y,\alpha}_{t_k} f(x) + (1-Y^{(k),t,y,\alpha}_{t_k} ) v_{k+1}(x,P_k(Y^{t,y,\alpha}_{t_k}))\right] & \leq & y_k f(x) + (1-y_k ) v_{k+1}(x,P_k(y))\nonumber\\
& \leq & w_k(t_k,x,y).\nonumber
\end{eqnarray}
But then because $\alpha$ was arbitrary, we conclude that
\[w_k(t,x,y) - w(t_k,x,y) \leq C\left|t_k-t\right|^{1/2}.\]

\vspace{5mm}\textit{Step 4:} Fix $(t,x,y)\in[t_{k-1},t_k]\times\mathbb{R}\times\Delta_k$ and $x'\in\mathbb{R}$. Let $\alpha\in\mathcal{A}$ be an arbitrary control for which $Y^{t,y,\alpha}_u \in \Delta_k$ for all $u\geq t$, almost-surely. Then we immediately compute
\begin{eqnarray}
w_k(t,x',y) & \geq & \mathbb{E}\left[Y^{(k),t,y,\alpha}_{t_k} f(X^{t,x'}_{t_k}) + (1-Y^{(k),t,y,\alpha}_{t_k} ) v_{k+1}(X^{t,x'}_{t_k},P_k(Y^{t,y,\alpha}_{t_k}))\right] \nonumber\\
& \geq & \mathbb{E}\left[Y^{(k),t,y,\alpha}_{t_k} f(X^{t,x}_{t_k}) + (1-Y^{(k),t,y,\alpha}_{t_k} ) v_{k+1}(X^{t,x}_{t_k},P_k(Y^{t,y,\alpha}_{t_k}))\right] - 2C\left|x-x'\right|,\nonumber
\end{eqnarray}
where $C>0$ is at least as large as the Lipschitz constant in $x$ of $f$ and $v_{k+1}$. Because $\alpha$ was arbitrary we conclude
\[w_k(t,x',y) \geq w(t,x,y) - 2C\left|x-x'\right|.\]

\vspace{5mm}\textit{Step 5:} Fix $(t,x,y)\in[t_{k-1},t_k]\times\mathbb{R}\times\Delta_k$. Suppose for now that $y$ is not a vertex of $\Delta_k$, or equivalently that the value in each coordinate is less than one. Denote by
\[\mathcal{I} := \left\{i\in\{1,\ldots,r\}\mid y_i = 0\right\},\qquad \mathcal{J} = \{1,\ldots,r\}\setminus\mathcal{I}\]
the disjoint collections of coordinates in which $y$ is zero and non-zero, respectively. Let
\[\delta := \min_{i\in\mathcal{J}}\left\{ \min\{|y_i|,|1-y_i|\} \right\} > 0.\]
Let $y'\in\Delta_k$ be any point satisfying $\|y-y'\|_{\ell^\infty}\leq\delta^2$. Let $\alpha\in\mathcal{A}$ be an arbitrary control for which $Y^{t,y,\alpha}_u \in \Delta_k$ for all $u\geq t$, almost-surely. Note that $\alpha$ equals zero almost surely for each coordinate in $\mathcal{I}$.

Define a new control $\alpha'\in\mathcal{A}$ as
\[\alpha'_u := (1-\|y-y'\|_{\ell^\infty}^{1/2})\,\alpha_u\]
and note that by construction we have
\[\min(Y^{(i),t,y,\alpha'}_u,\,1-Y^{(i),t,y,\alpha'}_u) \geq  \delta \|y-y'\|_{\ell^\infty}^{1/2}\]
for each $i\in\mathcal{J}$. Similarly, we have
\[\min(Y^{(i),t,y',\alpha'}_u,\,1-Y^{(i),t,y',\alpha'}_u) \geq \delta \|y-y'\|_{\ell^2}^{1/2} - \|y-y'\|_{\ell^2} \geq (\delta - \|y-y'\|_{\ell^2}^{1/2})\|y-y'\|_{\ell}^{1/2} \geq 0,\]
for all $u\geq t$, almost-surely. This together with the observation that $\alpha'$ equals zero in each direction in $\mathcal{I}$ implies that $Y^{t,y',\alpha'}_u\in\Delta_k$ for all $u\geq t$ almost-surely. Furthermore, we have
\begin{eqnarray}\label{Equation:HolderProof1}
\|Y^{t,y,\alpha}_{t_k}-Y^{t,y',\alpha'}_{t_k}\|_{\ell^\infty} & \leq & \|y-y'\|_{\ell^\infty} + \left(1-\sqrt{1-\|y-y'\|_{\ell^\infty}^{1/2}}\right)\|Y^{t,y,\alpha}_{t_k}-y\|_{\ell^\infty}\nonumber\\
& \leq &  \|y-y'\|_{\ell^\infty} + R\|y-y'\|_{\ell^\infty}^{1/2}\nonumber\\
& \leq & (1+R)\|y-y'\|_{\ell^\infty}^{1/2},
\end{eqnarray}
almost-surely, where $R>0$ is the diameter of the set $\Delta_k$.

Before proceeding with concrete bounds, we note an estimate regarding the perspective map. For any $y,y'\in\Delta_k$ such that $y_k,y'_k\neq 1$, we have
\[\|P_k(y) - P_k(y')\|_{\ell^\infty} \leq \left(\frac{1}{1-y_k}+\frac{1}{1-y'_k}\right)\|y-y'\|_{\ell^\infty}.\]
That is, the perspective map fails to be Lipschitz as $y_k,y'_k\to 1$. Using this along with the H\"older continuity and bounds on $v_{k+1}$ from Proposition~\ref{Proposition:LipschitzEstimate}, we can carefully bound
\[\begin{array}{l}
\left|(1-y_k)v_{k+1}(x,P_k(y)) - (1-y'_k)v_{k+1}(x,P_k(y'))\right| \\ 
\hspace{1cm} \leq \min\{1-y_k,1-y'_k\}\left|v_{k+1}(x,P_k(y))-v_{k+1}(x,P_k(y'))\right| \\
\hspace{2cm} + |y_k-y'_k|\left(\left|v_{k+1}(x,P_k(y))\right| + \left|v_{k+1}(x,P_k(y'))\right|\right)\\
\hspace{1cm} \leq \sqrt{2}C\|y-y'\|_{\ell^\infty}^{1/2} + 2C(1+|x|)\|y-y'\|_{\ell^\infty}.
\end{array}\]
This bound is easily seen in the case $y_k,y'_k\neq 1$ and may be carefully checked when either equals zero exactly.

Then using the bound above as well as the growth bounds on $f$, we check
\begin{eqnarray}
w_k(t,x,y') & \geq & \mathbb{E}\left[Y^{(k),t,y',\alpha'}_{t_k} f(X^{t,x}_{t_k}) + (1-Y^{(k),t,y',\alpha'}_{t_k} ) v_{k+1}(X^{t,x}_{t_k},P_k(Y^{t,y',\alpha'}_{t_k}))\right] \nonumber\\
& \geq & \mathbb{E}\left[Y^{(k),t,y,\alpha}_{t_k} f(X^{t,x}_{t_k}) + (1-Y^{(k),t,y,\alpha}_{t_k} ) v_{k+1}(X^{t,x}_{t_k},P_k(Y^{t,y,\alpha}_{t_k}))\right] \nonumber\\
& & \hspace{1cm} - \mathbb{E}\left[ \sqrt{2}C\|Y^{t,y,\alpha}_{t_k}-Y^{t,y',\alpha'}_{t_k}\|_{\ell^\infty}^{1/2} + 3C|X^{t,x}_{t_k}|\|Y^{t,y,\alpha}_{t_k}-Y^{t,y',\alpha'}_{t_k}\|_{\ell^\infty}\right].\nonumber
\end{eqnarray}
Applying H\"older's Inequality and the almost-sure bound in \eqref{Equation:HolderProof1}, we bound the last term by
\[\mathbb{E}\left[|X^{t,x}_{t_k}|\,\|Y^{t,y,\alpha}_{t_k}-Y^{t,y',\alpha'}_{t_k}\|_{\ell^\infty}\right] \leq \left(x^2+t_r\right)^{1/2}(1+R)^{1/2}\|y-y'\|_{\ell^\infty}^{1/4}.\]
Putting these all together and recalling that $\alpha$ was arbitrary, we conclude
\[w_k(t,x,y)-w_k(t,x,y') \leq C(1+\left|x\right|)\|y-y'\|_{\ell^2}^{1/4},\]
for large enough constant $C>0$.

If $y$ is a vertex, then the only admissible control is $\alpha\equiv 0$, so we obtain the same bound (in fact a better bound) for any nearby $y'$ directly from Proposition~\ref{Proposition:LipschitzEstimate}.

\vspace{5mm}\textit{Step 6:} Fix some $\delta > 0$ at least small enough that $\delta < 1/(2r)$. Fix $(t,x,y)\in[t_{k-1},t_k]\times\mathbb{R}\times\Delta_k$ such that at least one element of $(y_k,\ldots,y_r)$ is in $(0,\delta^2)$. Denote by
\[\mathcal{I} := \left\{i\in\{1,\ldots,r\}\mid y_i = 0\right\},\qquad \mathcal{J} = \left\{i\in\{1,\ldots,r\}\mid y_i \in (0,\delta^2)\right\},\]
\[\mathcal{K} := \{1,\ldots,r\}\setminus(\mathcal{I}\cup\mathcal{J})\]
the disjoint collections of coordinates in which $y$ is zero, ``small'', and ``large'', respectively. Let $y'\in\Delta_k$ be any point obtained from setting elements of $y$ in $\mathcal{J}$ to zero and adding these values to a single index $\kappa\in\mathcal{K}$. Then $y'$ is zero in all coordinates $\mathcal{I}\cup\mathcal{J}$ and non-zero (and ``large'') in all coordinates $\mathcal{K}$. Furthermore, $\|y-y'\|_{\ell^\infty} \geq \delta^2$, so
\[\delta \leq \|y-y'\|^{1/2}_{\ell^\infty}.\]

Let $\alpha\in\mathcal{A}$ be an arbitrary control for which $Y^{t,y,\alpha}_u \in \Delta_k$ for all $u\geq t$, almost-surely. Note that $\alpha$ equals zero almost surely for each coordinate in $\mathcal{I}$. Similarly, because each component of $Y^{t,y,\alpha}$ is a martingale, we conclude
\[\delta^2 \geq y_i = \mathbb{E}\left[Y^{(i),t,y,\alpha}_{t_k}\right] \geq \delta \mathbb{P}\left[Y^{(i),t,y,\alpha}_{t_k}\right]\]
for each $i\in\mathcal{J}$. That is, $\mathbb{P}\left[Y^{(i),t,y,\alpha}_{t_k}\geq\delta\right]\leq\delta$, so the $i$th coordinate of $Y^{t,y,\alpha}$ stays small with high probability.

Define $\alpha'\in\mathcal{A}$ by moving the values of $\alpha$ at coordinates $i\in\mathcal{J}$ to the $\kappa$th coordinate. Note that, by construction,
\[Y^{(\kappa),t,y',\alpha'}_u = Y^{(\kappa),t,y,\alpha}_u + \sum\limits_{i\in\mathcal{J}} Y^{(i),t,y,\alpha}_u \in [0,1]\]
and $Y^{(\kappa),t,y',\alpha'}_u\in\Delta_k$ for all $u\geq t$, almost-surely. Furthermore, we have
\[Y^{(i),t,y',\alpha'}_{t_k} = Y^{(i),t,y,\alpha}_{t_k},\]
for each $i\in\mathcal{I}\cup\mathcal{K}\setminus\{\kappa\}$, and
\[\mathbb{P}\left[\left| Y^{(i),t,y',\alpha'}_{t_k} - Y^{(i),t,y,\alpha}_{t_k}\right| \geq \delta\right] \leq \delta\]
for each $i\in\mathcal{J}$. Lastly, we have
\[\mathbb{P}\left[\left| Y^{(\kappa),t,y',\alpha'}_{t_k} - Y^{(\kappa),t,y,\alpha}_{t_k}\right| \geq r\delta\right] \leq \delta.\]
In summary, we have
\[\mathbb{P}\left[\|Y^{t,y',\alpha'}_{t_k}-Y^{t,y,\alpha}_{t_k}\|_{\ell^\infty}\geq 2r\delta\right] \leq \delta.\]

Then using the same bounds as in the previous step, we can now compute
\begin{eqnarray}
w_k(t,x,y') & \geq & \mathbb{E}\left[Y^{(k),t,y',\alpha'}_{t_k} f(X^{t,x}_{t_k}) + (1-Y^{(k),t,y',\alpha'}_{t_k} ) v_{k+1}(X^{t,x}_{t_k},P_k(Y^{t,y',\alpha}_{t_k}))\right] \nonumber\\
& \geq & \mathbb{E}\left[Y^{(k),t,y,\alpha}_{t_k} f(X^{t,x}_{t_k}) + (1-Y^{(k),t,y,\alpha}_{t_k} ) v_{k+1}(X^{t,x}_{t_k},P_k(Y^{t,y,\alpha}_{t_k}))\right] \nonumber\\
& & - \mathbb{E}\left[ \sqrt{2}C\|Y^{t,y,\alpha}_{t_k}-Y^{t,y',\alpha'}_{t_k}\|_{\ell^\infty}^{1/2} + 3C|X^{t,x}_{t_k}|\|Y^{t,y,\alpha}_{t_k}-Y^{t,y',\alpha'}_{t_k}\|_{\ell^\infty}\right].\nonumber
\end{eqnarray}
The first term on the right-hand-side may be bounded as
\begin{eqnarray}
\mathbb{E}\left[\|Y^{t,y,\alpha}_{t_k}-Y^{t,y',\alpha'}_{t_k}\|^{1/2}_{\ell^\infty}\right] & \leq & \sqrt{2}\,\mathbb{P}\left[\|Y^{t,y',\alpha'}_{t_k}-Y^{t,y,\alpha}_{t_k}\|_{\ell^\infty}\geq 2r\delta\right] + \sqrt{2r\delta} \nonumber\\
& \leq & \sqrt{2(r+1)}\delta^{1/2}.\nonumber
\end{eqnarray}
Similarly, the second term may be bounded as
\begin{eqnarray}
\mathbb{E}\left[|X^{t,x}_{t_k}|\|Y^{t,y,\alpha}_{t_k}-Y^{t,y',\alpha'}_{t_k}\|_{\ell^\infty}\right] & \leq & \sqrt{|x|^2+t_r}\,\mathbb{E}\left[\|Y^{t,y,\alpha}_{t_k}-Y^{t,y',\alpha'}_{t_k}\|^2_{\ell^\infty}\right]^{1/2}\nonumber\\
& \leq & 2\sqrt{(|x|^2+t_r)(1+r^2)}\delta.\nonumber
\end{eqnarray}
Putting these all together and recalling that $\alpha$ was arbitrary and $\delta \leq \|y-y'\|_{\ell^\infty}^{1/2}$, we conclude
\[w_k(t,x,y)-w_k(t,x,y') \leq C(1+\left|x\right|)\|y-y'\|^{1/4}_{\ell^2},\]
for large enough constant $C>0$.

\vspace{5mm}\textit{Step 7:} We now briefly remark how to put all of these estimates together. We consider the H\"older estimates in each coordinate separately as they can be combined in the end using triangle inequality. Note that the Lipschitz regularity in $x$ has already been proven.

In the time direction, fix $(t,x,y)\in[t_{k-1},t_k]\times\mathbb{R}\times\Delta_k$ and $t'\in[t_{k-1},t_k]$ some small $\theta>0$ such that $\theta<t_k-t_{k-1}$. If $t'\leq t$, then by Step 1 we have
\[w_k(t,x,y) - w_k(t',x,y) \leq 2C^2\left|t-t'\right|^{1/2}.\]
If $t<t'=t_k$, then by Step 3 we have
\[w_k(t,x,y) - w_k(t',x,y) = w_k(t,x,y) - w_k(t_k,x,y) \leq C\left|t_k-t\right|^{1/2} = C\left|t-t'\right|^{1/2}.\]
Now suppose that $t < t' \leq t_k-|t-t'|^{1/2}$. By Step 2, we have
\[w_k(t,x,y) - w_k(t',x,y) \leq 4C^2\frac{|t-t'|}{t_k-t'} \leq 4C^2\frac{|t-t'|}{|t-t'|^{1/2}} = 4C^2|t-t'|^{1/2}.\]
In the next step, we critically see where the $(1/4)$-H\"older coefficient appears. If $t_k-|t-t'|^{1/2} \leq t < t' < t_k$ then by an application of Step 1 and Step 3, we see
\begin{eqnarray}
w_k(t,x,y) - w_k(t',x,y) & \leq & [w_k(t,x,y) - w_k(t_k,x,y)] + [w_k(t_k,x,y)-w_k(t',x,y)] \nonumber\\
& \leq & C|t_k-t| + 4C^2|t_k-t'|^{1/2} \nonumber\\
& \leq & C|t-t'|^{1/2} + 4C^2|t-t'|^{1/4}.\nonumber
\end{eqnarray}
Lastly we consider the case $t \leq t_k-|t-t'|^{1/2} \leq t' < t_k$. By an application of each of Steps 1 through 3, we see
\begin{eqnarray}
w_k(t,x,y) - w_k(t',x,y) & \leq & [w_k(t,x,y) - w_k(t_k-|t-t'|^{1/2},x,y)] \nonumber\\
&& + [w_k(t_k-|t-t'|^{1/2},x,y)-w_k(t_k,x,y)] + [w_k(t_k,x,y) - w_k(t',x,y)] \nonumber\\
& \leq &  4C^2\frac{|t_k-|t-t'|^{1/2}-t|}{|t-t'|^{1/2}} + C|t-t'|^{1/4} + 4C^2|t_k-t'|^{1/2} \nonumber\\
& \leq & 4C^2\frac{|t-t'|}{|t-t'|^{1/2}} + C|t-t'|^{1/4} + 4C^2|t-t'|^{1/4}\nonumber\\
& \leq & (C+8C^2)|t-t'|^{1/4}.\nonumber
\end{eqnarray}
Of course, by taking a large enough constant we can bound all $|t-t'|^{1/2}$ terms by $|t-t'|^{1/4}$ terms and obtain the $(1/4)$-H\"older continuity result in $t$.

The $(1/4)$-H\"older continuity result follows by a similar approach by cases as in the time perturbation case. The key idea is that Step 5 and Step 6 tell locally how to perturb in a $(1/4)$-H\"older way, including onto and off the boundaries. Then by a covering argument and the compactness of $\Delta_k$ we can obtain a finite chain of local H\"older inequalities connecting any two points and obtain the result for sufficiently large constant.
\end{proof}

\section{Proof of Theorem~\ref{Thm:TimeDependentDPP}}\label{Appendix:TimeDependentDPP}

This argument is essentially a time-dependent version of that given in the proof of Lemma~\ref{Lem:WeakDPPLemma}. The key idea here is to use the convexity of the set $\Delta_k$ and the concavity of $w_k$ in $y$ to construct an $\epsilon$-suboptimal control which satisfies the state-constraint as a convex combination of admissible controls starting from nearby points.

\begin{proof}
Fix $(t,x,y)\in[t_{k-1},t_k)\times\mathbb{R}\times\Delta_k$ and $0<h<t_k-t_{k-1}$. For convenience of notation, define $\theta := t_k$ and
\begin{equation}\nonumber\begin{array}{rccl}
A & := & \sup\limits_{\alpha\in\mathcal{A}_t} & \mathbb{E}\left[w_k\left(\tau^\alpha,X^{t,x}_{\tau^\alpha},Y^{t,y,\alpha}_{\tau^\alpha}\right)\right] \\
& & \text{s.t.} & Y^{t,y,\alpha}_u \in \Delta_k\text{ for all }u\geq t.
\end{array}\end{equation}
The inequality $w_k(t,x,y) \leq A$ is a standard result even in the case of these state-constraints. We refer the interested reader to Theorem 3.3 in \cite{Touzi2013} and instead focus on the opposite inequality.


\vspace{5mm}\textit{Step 1:} Fix an arbitrary $\epsilon > 0$. Choose $R>0$ large enough that
\[\mathbb{P}\left[\sup\limits_{t\leq u\leq t+h} \left| W_u-W_t\right| \geq R\right] \leq \epsilon^2.\]
Because $w_k$ is continuous on the compact set $[t,t+h]\times[x-R,x+R]\times\Delta_k$, we can find $\delta > 0$ small enough that
\[\left|w_k(t',x',y')-w_k(t',x',y'')\right|\leq\epsilon\]
for all $(t',x')\in[t,t+h]\times[x-R,x+R]$ and $y',y''\in\Delta_k$ such that
\[\|y'-y''\|_{\ell^\infty} \leq \delta.\]
Similarly, because $f$ is Lipschitz and $v_{k+1}$ is Lipschitz in $x$ uniformly in $y$, we can find $\delta > 0$, possibly smaller than before, such that we also have
\[\left|f(x')-f(x'')\right|+\left|v_{k+1}(x',y')-v_{k+1}(x'',y')\right|\leq\epsilon\]
for all $x',x''\in\mathbb{R}$ and $y'\in\Delta_{k+1}$ such that
\[\left|x'-x''\right| \leq \delta.\]
Finally, take $\delta > 0$ potentially even smaller so that
\[\delta^{1/2}+\delta^{1/4} \leq \epsilon.\]

\vspace{5mm}\textit{Step 2:} We first construct a finite mesh on $[t,t+h]$ and $\Delta_k$ which will be fine enough to take advantage the continuity of $w_k$. Let $\Lambda := \{t_i\}_{i=1}^M$ be a finite collection of mesh points in $[t,t+h]$ with the key property that for any $u\in[t,t+h]$, there exists $i\in\{1,\ldots,M\}$ such that $u \leq t_i \leq u+\delta$. 

By the compactness and convexity of $\Delta_k$, we can obtain a finite subset of $\Delta_k$, $\mathcal{P} := \{y_\ell\}_{\ell=1}^P$, with the property that
\begin{itemize}
\item The convex hull of $\mathcal{P}$ is $\Delta_k$, and
\item Any point $y\in\Delta_k$ can be written as a convex combination of points in $\mathcal{P}$, each contained in a $\delta$-neighborhood of $y$.
\end{itemize}
In particular, we can find a continuous function $T : \Delta_k \to [0,1]^P$ with the properties
\begin{itemize}
\item $T_\ell(y) = 0$ for all $y\in\Delta_k$ such that $|y-y_\ell|>\delta$
\item $\sum_{\ell=1}^P T_\ell(y) = 1$ for all $y\in\Delta_k$, and
\item $\sum_{\ell=1}^P y_\ell T_\ell(y) = y$ for all $y\in\Delta_k$.
\end{itemize}
This corresponds to a continuous map from a point $y\in\Delta_k$ to a probability weighting of points in $\mathcal{P}$ such that $y$ is a convex combination of nearby points in $\mathcal{P}$. 

By the same type of covering argument as in the proof of Lemma~\ref{Lem:ControlCharacterization}, we can obtain a finite and disjoint covering of $[x-R,x+R]$ by measurable sets $\{A_j\}_{j=1}^N$, each contained in a $\delta$-ball, and controls $\alpha_{ij\ell}\in\mathcal{A}_{t_i}$ with the key properties that $Y^{t_i,y_\ell,\alpha_{ij\ell}}_u \in \Delta_k$ for all $u\geq t_i$ and
\[\mathbb{E}\left[Y^{(k),t_i,y_\ell,\alpha_{ij\ell}}_\theta f(X^{t_i,x}_\theta) + (1-Y^{(k),t_i,y_\ell,\alpha_{ij\ell}}_\theta)v_{k+1}(X^{t_i,x}_\theta,P_k(Y^{t_i,y_\ell,\alpha_{ij\ell}}_\theta))\right] \geq w_k(t_i,x,y_\ell) - 3\epsilon\]
for each $i\in\{1,\ldots,M\}$, $j\in\{1,\ldots,N\}$, $\ell\in\{1,\ldots,P\}$, and $x\in A_j$.

\vspace{5mm}\textit{Step 3:} Fix an arbitrary control $\alpha\in\mathcal{A}$ for which $Y^{t,y,\alpha}_u \in \Delta_k$ for all $u\geq t$ and let $\tau^\alpha$ be the associated stopping time which is valued in $[t,t+h]$. We are next going to construct a new control related to the suboptimal controls $\alpha_{ij\ell}$. In words, we will follow $\alpha$ up to the stopping time $\tau^\alpha$, then set the control to zero until the first subsequent hitting time of $\Lambda$. Then we will follow an appropriate convex combination of the controls $\alpha_{ij\ell}$.

To make this precise, define a stopping time
\[\overline\tau := \inf\{t\geq\tau^\alpha\mid t \in \Lambda\}.\]
Define a collection of controls $\alpha_\ell \in \mathcal{A}$ as
\[\alpha_{\ell,u} := 1_{\{u\in[t,\tau^\alpha]\}} \alpha_u + 1_{\{u > \overline\tau\}}\sum\limits_{i=1}^M \sum\limits_{j=1}^N 1_{\{\overline\tau = t_i\}}1_{\{X^{t,x}_{\overline\tau}\in A_j\}}\alpha_{ij\ell,u}\]
for each $\ell \in \{1,\ldots,P\}$ and all $u\geq t$. Finally, define a control $\overline\alpha\in\mathcal{A}$ as
\[\overline\alpha_u := 1_{\{u\in[t,\tau^\alpha]\}} \alpha_u + 1_{\{u > \overline\tau\}}\sum\limits_{\ell=1}^P T_\ell(Y^{t,y,\alpha}_{\tau^\alpha})\alpha_{\ell,u}\]
for all $u\geq t$.

The control $\overline\alpha$ is adapted because the map $T$ is continuous. Similarly, it can be easily seen to be square-integrable. The key property, however, is that $\overline\alpha$ satisfies $Y^{t,y,\overline\alpha}_u \in \Delta_k$ for all $u\geq t$. In words, this follows from the convexity of the set $\Delta_k$ and the fact that $\overline\alpha$ is a convex combination of controls, each of which satisfy the state-constraint.

Making this precise, we use the assumed properties of the map $T$ and the dynamics of $Y$ to compute 
\begin{eqnarray}
Y^{t,y,\overline\alpha}_\theta & = & y + \int_t^{\tau^\alpha} \alpha_u dW_u + \int_{\overline\tau}^\theta \overline\alpha_u dW_u \nonumber\\
& = & Y^{t,y,\alpha}_{\tau^\alpha} + \sum\limits_{i,j,\ell} 1_{\{\overline\tau=t_i\}}1_{\{X^{t,x}_{\overline\tau}\in A_j\}} T_\ell(Y^{t,y,\alpha}_{\tau^\alpha})\int_{\overline\tau}^\theta \alpha_{ij\ell,u}dW_u \nonumber\\
& = & 1_{\{\left|W_{\overline\tau}-W_t\right|\geq R\}}Y^{t,y,\alpha}_{\tau^\alpha} + \sum\limits_{i,j} 1_{\{\overline\tau=t_i\}}1_{\{X^{t,x}_{\overline\tau}\in A_j\}} \left(Y^{t,y,\alpha}_{\tau^\alpha}+\sum\limits_{\ell} T_\ell(Y^{t,y,\alpha}_{\tau^\alpha})\int_{\overline\tau}^\theta \alpha_{ij\ell,u}dW_u\right) \nonumber\\
& = & 1_{\{\left|W_{\overline\tau}-W_t\right|\geq R\}}Y^{t,y,\alpha}_{\tau^\alpha} + \sum\limits_{i,j} 1_{\{\overline\tau=t_i\}}1_{\{X^{t,x}_{\overline\tau}\in A_j\}} \sum\limits_{\ell}T_\ell(Y^{t,y,\alpha}_{\tau^\alpha})\left(y_\ell+\int_{\overline\tau}^\theta \alpha_{ij\ell,u}dW_u\right) \nonumber\\
& = & 1_{\{\left|W_{\overline\tau}-W_t\right|\geq R\}}Y^{t,y,\alpha}_{\tau^\alpha} + \sum\limits_{i,j} 1_{\{\overline\tau=t_i\}}1_{\{X^{t,x}_{\overline\tau}\in A_j\}} \sum\limits_{\ell}T_\ell(Y^{t,y,\alpha}_{\tau^\alpha})Y^{t_i,y_\ell,\alpha_{ij\ell}}_\theta \nonumber
\end{eqnarray}
Recall though that $Y^{t_i,y_\ell,\alpha_{ij\ell}}_\theta \in \Delta_k$ and $Y^{t,y,\alpha}_{\tau^\alpha} \in \Delta_k$ almost-surely. Then the equality above and the convexity of $\Delta_k$ demonstrate that $Y^{t,y,\overline\alpha}_u \in \Delta_k$ for all $u\geq t$ almost-surely.

\vspace{5mm}\textit{Step 4:} We now proceed to make a very delicate series of estimates. First, we have
\begin{eqnarray}
w_k(t,x,y) & \geq & \mathbb{E}\left[Y^{(k),t,y,\overline\alpha}_\theta f(X^{t,x}_\theta) + (1-Y^{(k),t,y,\overline\alpha}_\theta)v_{k+1}(X^{t,x}_\theta,P_k(Y^{t,y,\overline\alpha}_\theta))\right] \nonumber\\
& \geq & \sum\limits_{i,j}\mathbb{E}\left[1_{\{\overline\tau = t_i\}}1_{\{X^{t,x}_{\overline\tau}\in A_j\}}\left(Y^{(k),t,y,\overline\alpha}_\theta f(X^{t,x}_\theta) + (1-Y^{(k),t,y,\overline\alpha}_\theta)v_{k+1}(X^{t,x}_\theta,P_k(Y^{t,y,\overline\alpha}_\theta))\right)\right] \nonumber\\
& & \hspace{1cm} - C\mathbb{E}\left[1_{\{\left|W_{\overline\tau}\right|\geq R\}}\left(1+\left|X^{t,x}_{\overline\tau}\right|\right)\right]\nonumber,
\end{eqnarray}
where the last term comes from known growth bounds on $f$ and $v_{k+1}$. Of course, this term is bounded by $\epsilon\sqrt{C(1+|x|)}$ by the choice of $R$ and use of H\"older's Inequality. Next, rewriting each term in the sum above using the key property from the construction of $\overline\alpha$ in the previous step, we see
\[\begin{array}{l}
\mathbb{E}\left[1_{\{\overline\tau = t_i\}}1_{\{X^{t,x}_{\overline\tau}\in A_j\}}\left(Y^{(k),t,y,\overline\alpha}_\theta f(X^{t,x}_\theta) + (1-Y^{(k),t,y,\overline\alpha}_\theta)v_{k+1}(X^{t,x}_\theta,P_k(Y^{t,y,\overline\alpha}_\theta))\right)\right] \\
\hspace{1cm} = \mathbb{E}\left[1_{\{\overline\tau = t_i\}}1_{\{X^{t,x}_{\overline\tau}\in A_j\}}\left(\sum_{\ell=1}^P T_\ell(Y^{t,y,\alpha}_{\tau^\alpha}) Y^{(k),t_i,y_\ell,\alpha_{ij\ell}}_\theta f(X^{t,x}_\theta) \right.\right. \\
\hspace{2cm}\left.\left. + (1-\sum_{\ell=1}^P T_\ell(Y^{t,y,\alpha}_{\tau^\alpha})Y^{(k),t_i,y_\ell,\alpha_{ij\ell}}_\theta)v_{k+1}(X^{t,x}_\theta,P_k(\sum_{\ell=1}^P T_\ell(Y^{t,y,\alpha}_{\tau^\alpha})Y^{t_i,y_\ell,\alpha_{ij\ell}}_\theta))\right)\right]\\
\hspace{1cm} \geq \sum\limits_{\ell=1}^P \mathbb{E}\left[1_{\{\overline\tau = t_i\}}1_{\{X^{t,x}_{\overline\tau}\in A_j\}}T_\ell(Y^{t,y,\alpha}_{\tau^\alpha})\left( Y^{(k),t_i,y_\ell,\alpha_{ij\ell}}_\theta f(X^{t,x}_\theta) \right.\right.\\
\hspace{2cm}\left.\left. + (1-Y^{(k),t_i,y_\ell,\alpha_{ij\ell}}_\theta)v_{k+1}(X^{t,x}_\theta,P_k(Y^{t_i,y_\ell,\alpha_{ij\ell}}_\theta))\right)\right]
\end{array}\]
using the concavity of $v_{k+1}$ composed with the perspective function in $y$ (See the proof of Proposition~\ref{Proposition:Concavity}). Next, by the suboptimality conditions of $\alpha_{ijk}$, we see
\[\begin{array}{l}
 \mathbb{E}\left[1_{\{\overline\tau = t_i\}}1_{\{X^{t,x}_{\overline\tau}\in A_j\}}T_\ell(Y^{t,y,\alpha}_{\tau^\alpha})\left( Y^{(k),t_i,y_\ell,\alpha_{ij\ell}}_\theta f(X^{t,x}_\theta)  + (1-Y^{(k),t_i,y_\ell,\alpha_{ij\ell}}_\theta)v_{k+1}(X^{t,x}_\theta,P_k(Y^{t_i,y_\ell,\alpha_{ij\ell}}_\theta))\right)\right]\\
\hspace{1cm} \geq \mathbb{E}\left[1_{\{\overline\tau = t_i\}}1_{\{X^{t,x}_{\overline\tau}\in A_j\}}T_\ell(Y^{t,y,\alpha}_{\tau^\alpha})w_k(t_i,X^{t,x}_{\overline\tau},y_\ell)\right] - 3\epsilon \\
\hspace{1cm} \geq \mathbb{E}\left[1_{\{\overline\tau = t_i\}}1_{\{X^{t,x}_{\overline\tau}\in A_j\}}T_\ell(Y^{t,y,\alpha}_{\tau^\alpha})w_k(t_i,X^{t,x}_{\overline\tau},Y^{t,y,\alpha}_{\tau^\alpha})\right] - 4\epsilon,
\end{array}\]
where we used the locality property of the map $T$ and continuity of $w_k$ assumed in the construction of $\mathcal{P}$. Lastly, summing over $i,j,\ell$, we see
\[\begin{array}{l}
\sum_{i,j,\ell} \mathbb{E}\left[1_{\{\overline\tau = t_i\}}1_{\{X^{t,x}_{\overline\tau}\in A_j\}}T_\ell(Y^{t,y,\alpha}_{\tau^\alpha})w_k(t_i,X^{t,x}_{\overline\tau},Y^{t,y,\alpha}_{\tau^\alpha})\right] \\
\hspace{1cm} \geq \mathbb{E}\left[w_k(\overline\tau, X^{t,x}_{\overline\tau}, Y^{t,y,\alpha}_{\tau^\alpha})\right] - \mathbb{E}\left[1_{\{\left(|W_{\overline\tau}\right|\geq R\}}w_k(\overline\tau, X^{t,x}_{\overline\tau}, Y^{t,y,\alpha}_{\tau^\alpha}\right]\\
\hspace{1cm} \geq \mathbb{E}\left[w_k(\tau^\alpha, X^{t,x}_{\tau^\alpha}, Y^{t,y,\alpha}_{\tau^\alpha})\right] - C\,\mathbb{E}\left[|\overline\tau-\tau^\alpha|^{1/4} + \left|X^{t,x}_{\overline\tau}-X^{t,x}_{\tau^\alpha}\right| + 1_{\{\left|W_{\overline\tau}\right|\geq R\}}\left(1+\left|X^{t,x}_{\overline\tau}\right|\right)\right] \\
\hspace{1cm} \geq \mathbb{E}\left[w_k(\tau^\alpha, X^{t,x}_{\tau^\alpha}, Y^{t,y,\alpha}_{\tau^\alpha})\right] - C\left(\delta^{1/4} + \delta^{1/2} + \epsilon\left(1+|x|\right)\right]
\end{array}\]
for sufficiently large $C>0$. In this step, we used growth bounds on $w_k$ and the H\"older estimates from Proposition~\ref{Proposition:TimeDependentHolderEstimates} together with the fact that $\left|\overline\tau-\tau^\alpha\right|\leq\delta$ by construction. By the choice of $\delta$, however, we see this last error term is bounded by $ C(2+|x|)\epsilon$.

Putting all these computations together and recalling that $\alpha$ and $\epsilon>0$ were arbitrary, we see
\[w_k(t,x,y) \geq A.\]
\end{proof}

{\small 

\bibliographystyle{siam}
\bibliography{Bibliography}
}
\end{document}